\documentclass[11pt,a4paper]{article}
\usepackage{a4}
\usepackage{amsmath}
\usepackage{amssymb}
\usepackage{graphicx}
\usepackage{enumerate}
\def\+{\oplus}

\newcommand{\dt}{\Delta t}

\newcommand{\R}{{\mathbb R}}

\newcommand{\N}{{\mathbb N}}

\newcommand{\ch}{{\mathcal H}}

\newcommand{\cG}{{\mathcal G}}

\newcommand{\cI}{{\mathcal I}}

\newcommand{\cC}{{\mathcal C}}

\newcommand{\cT}{{\mathcal T}}

\newcommand{\cK}{{\mathcal K}}

\newcommand{\wtM}{{\tilde m}}
\newcommand{\wtU}{{\tilde u}}

\newcommand{\diver}{{\rm{div}}}
\newcommand{\T}{{\mathbb{T}}}

\newcommand{\ds}{\displaystyle}
\def\squareforqed{\hbox{\rlap{$\sqcap$}$\sqcup$}}
\def\qed{\ifmmode\else\unskip\quad\fi\squareforqed}
\def\smartqed{\def\qed{\ifmmode\squareforqed\else{\unskip\nobreak\hfil
\penalty50\hskip1em\null\nobreak\hfil\squareforqed
\parfillskip=0pt\finalhyphendemerits=0\endgraf}\fi}}

\newenvironment{proof}[1][Proof]{\textbf{#1.} }{\ \rule{0.5em}{0.5em}}
\newtheorem{remark}{\textbf{Remark}}
\newtheorem{lemma}{\textbf{Lemma}}
\newtheorem{theorem}{\textbf{Theorem}}

\newtheorem{acknowledgement}{\textbf{Acknowledgement}}

\title{Mean field games: convergence of a finite difference method}
 \author{Yves Achdou \thanks { Univ. Paris Diderot, Sorbonne Paris Cit{\'e}, Laboratoire Jacques-Louis Lions, UMR 7598, UPMC, CNRS, F-75205 Paris, France.
  achdou@ljll-univ-paris-diderot.fr},
 Fabio Camilli  \thanks{Dipartimento di Metodi e Modelli Matematici per le Scienze Applicate, Universit{\`a} Roma "La Sapienza"
 Via Antonio Scarpa  16, I-00161 Roma, camilli@dmmm.uniroma1.it},
 Italo Capuzzo-Dolcetta \thanks{Dipartimento di Matematica, Universit{\`a} Roma "La Sapienza", Piazzale A. Moro 2, I-00185 Roma, capuzzo@mat.uniroma1.it}
 }
\begin{document}

\maketitle
\begin{abstract}
Mean field type models describing the limiting behavior,  as the number of players tends to $+\infty$, of stochastic differential game problems, have been recently  introduced by J-M.
Lasry and P-L. Lions. Numerical methods for the approximation of the stationary and evolutive versions of such models have been proposed by the authors in previous works .
Convergence theorems for these methods are proved under various assumptions.
\end{abstract}



 \begin{description}
 \item [{\bf MSC 2000}:] 65M06,65M12,91-08,91A23,49L25.
 \item [{\bf Keywords}:] Mean field games, numerical methods.
 \end{description}

\section {Introduction}
Mean field type models describing the asymptotic behavior of  stochastic differential games (Nash equilibria) as the number of players
tends to $+\infty$ have recently been introduced by J-M. Lasry and P-L. Lions
 \cite{MR2269875,MR2271747,MR2295621}.  They may lead to systems of  evolutive partial differential equations involving  two unknown scalar functions: the density of the agents in a given state $x$,
namely  $m=m(t,x)$
 and the potential  $u=u(t,x)$.  Since  the present work is devoted to  finite difference schemes, we will assume that the dimension of the state space is $d=2$ (what follows could be generalized to any dimension $d$, although in practice, finite difference methods require too many computing resources when $d\ge 4$).
In the periodic setting, typical such  model comprises  the following system of evolution partial differential equations
  \begin{eqnarray}
  \label{eq:1}
\frac {\partial  u}{\partial t} (t,x) -\nu \Delta  u  (t,x) + H(x,\nabla  u (t,x))  = \Phi[m(t,\cdot)] (x), &\quad \hbox{in }(0,T)\times  \T^2, \\
\label{eq:2}
\frac {\partial  m}{\partial t}  (t,x)+\nu \Delta  m (t,x) +\diver\left(  m (t,\cdot) \frac {\partial H} {\partial p} (\cdot,\nabla  u(t,\cdot) )\right)(x)  = 0,&\quad \hbox{in } (0,T)\times  \T^2,
\end{eqnarray}with the initial and terminal conditions
\begin{equation}
  \label{eq:3} u(0,x)=u_0(x),\quad  m(T,x) = m_T(x),\quad \hbox{in } \T^2,
\end{equation}
given a   cost function $u_0$ and a probability density  $m_T$.
\\
Let us make some comments on the boundary value problem (\ref{eq:1})-(\ref{eq:3}).\\
First, note  that $t$ is the remaining time to the horizon, (the physical time is in fact $T-t$), so $u_0$ should be seen as a final cost or incitation,
 whereas $m_T$ is the density of the agents at the beginning of the game.
\\
Here, we denote by $\T^2=[0,1]^2$ the $2-$dimensional unit torus, by $\nu$ a nonnegative constant
 and by $\Delta$, $\nabla$ and $\diver$, respectively, the Laplace, the gradient
and the divergence operator acting on the $x$ variable.
The system also involves the scalar Hamiltonian $H(x,p)$, which is assumed to be convex with respect to $p$ and  $\cC^1$ regular w.r.t. $x$ and $p$.
 The notation $\frac{\partial H}{\partial p}(x,q)$ is used for the gradient of $p\mapsto H(x,p)$ at $p=q$.
\\
Finally, in  the cost term  $\Phi[m(t,\cdot)] (x)$,  $\Phi$ may be
\begin{itemize}
\item either a local operator, i.e. $\Phi[m(t,\cdot)] (x)= F(m(t,x))$ where $F$ is a  $\cC^1$ regular function defined on $\R_+$. In this case,
there are existence theorems of either classical (see \cite{porretta2012}) or weak solutions (see \cite{MR2271747}), under suitable assumptions on the data, $H$ and $F$.
\item or a non local operator  which continuously maps  the set of probability  measures on $\T^2$ (endowed with the weak * topology) to a bounded subset of $Lip(\T^2)$, the Lipschitz functions on $\T^2$,
 and for example maps continously $C^{k,\alpha}(\T^2)$ to  $C^{k+1,\alpha}(\T^2)$, for all $k\in \N$ and $0\le \alpha <1$. In this case, classical solutions of (\ref{eq:1})-(\ref{eq:3}) are shown to exist
under natural assumptions on the data and some technical assumptions on $H$.
\end{itemize}
We have chosen to focus on the case when the cost $u_{|t=0}$  depends directly on $x$. In some realistic situations,
the final cost may depend on the density of the players, i.e. $u_{|t=0}= \Phi_0[m_{|t=0}](x)$, where
 $\Phi_0$ is an operator acting on probability densities, which may be local or not.
This case can be handled by the methods presented below,  but we will not discuss it in the present work.\\ \\
System (\ref{eq:1})-~(\ref{eq:2}) consists of a forward Bellman equation coupled with a backward Fokker-Planck equation. The forward-backward structure is an important feature of
this system, which makes it necessary to design new strategies for its mathematical analysis (see \cite{MR2271747,MR2295621}) and  for numerical approximation.
\\ \\
The following steady state version of (\ref{eq:1})-(\ref{eq:3}) arises when mean field games with infinite horizon are considered (ergodic problem):
  \begin{eqnarray}
  \label{eq:4}
 -\nu \Delta  u  (x) + H(x,\nabla  u (x))  +\lambda  = \Phi[m(\cdot)] (x), &\quad \hbox{in } \T^2, \\
\label{eq:5}
-\nu \Delta  m (x) -\diver\left(  m\frac {\partial H} {\partial p} (\cdot,\nabla  u)\right) (x)  = 0,&\quad \hbox{in }  \T^2.
\end{eqnarray}
 with the additional normalization of $u$: $\int_{\T^2} u=0$. The unknowns in (\ref{eq:4})-(\ref{eq:5}) are the density $m$, the function $u$ and the scalar $\lambda$.
\\
We refer to the mentioned papers of  J-M. Lasry and P-L. Lions  for analytical results concerning problems (\ref{eq:1})-(\ref{eq:3})  and (\ref{eq:4})-(\ref{eq:5}) as well as for their interpretation in stochastic game theory. Let us only mention here that a very important feature of the mean field model above is that uniqueness and stability may be obtained under reasonable assumptions, see \cite{MR2269875,MR2271747,MR2295621}, in contrast with the Nash system describing the individual behavior of each player, for which uniqueness hardly occurs.
To be more precise, uniqueness for (\ref{eq:1})-(\ref{eq:3}) is true   if $\Phi$  is monotonous in the sense that for all probability measures $m$ and $\tilde m$ on $\T^2$,
\begin{equation}\label{eq:6}
  \int (\Phi[m](x)-\Phi[\tilde m] (x)) (dm(x)-d\tilde m(x)) \le 0\Rightarrow m=\tilde m.
\end{equation}
Examples of MFG models with applications in economics and social sciences are proposed in \cite{MR2762362}.
Many important aspects of the mathematical theory developed by  J-M. Lasry and P-L. Lions on MFG are not published in journals or books, but can be found in the videos of  the lectures of P-L. Lions at Coll{\`e}ge de France:
see the web site of Coll{\`e}ge de France, \cite{PLL}.\\
An important research activity is currently going on about approximation procedures of different types of mean field games models, see
 \cite{MR2647032} for a numerical method based  on the reformulation of the model as an optimal control problem for the Fokker-Planck equation
with an application in economics and \cite{MR2601334} for a  work on discrete time, finite state space mean  field games.
We also refer to \cite{gueant2011,gueant2012} for a specific constructive approach for quadratic Hamiltonians. Finally,
 a semi-discrete approximation for a first order mean field games problem has been studied in \cite{cam_silva}.\\
\\ In \cite{MR2679575} and \cite{ACCDplanning}, the authors have proposed and studied
  finite difference methods  basically relying  on monotone approximations of
 the Hamiltonian and on a suitable weak formulation of the Fokker-Planck equation, both for infinite  and finite horizon mean field games.
These schemes were shown to have  several important features:
\begin{itemize}
\item  existence and uniqueness for the discretized problems can be obtained by similar arguments as those used in the continuous case,
\item  they are robust when $\nu\to 0$ (the deterministic limit of the models),
\item  bounds on the solutions (especially on the Lipschitz norm of $u(t,\cdot)$), which are uniform in the grid step, may be proved under reasonable assumptions on the data.
\end{itemize}
Fast algorithms for solving the discrete nonlinear systems arising in the discrete version of the MFG systems have been proposed in \cite{achdou_perez_2012}.
In the present paper, we would like to discuss the convergence of the schemes proposed in \cite{MR2679575} in the reference case when $H(x,p)$ is of the form
\begin{equation}
  \label{eq:7}
H(x,p)= \ch(x) + |p|^\beta,
\end{equation}
 where $\beta$ is a positive number greater than $1$ and $\ch$ is a  periodic continuous function, under suitable monotony assumptions on the operator $\Phi$.
These assumptions lead to uniqueness for both the continuous and discrete system, and also to a priori and stability estimates under further assumptions.
\\
This work is organized as follows: in Section \ref{sec:finite-diff-schem}, we recall the finite difference schemes proposed in \cite{MR2679575,ACCDplanning}.
Section \ref{sec:basic-facts-numer-1} is devoted to basic facts concerning the discrete Hamiltonians and to the  fundamental identity which leads to uniqueness and stability.
When $\Phi$ is a local operator, a consequence of this key identity is the a priori estimates on the solutions of the discrete MFG system that is presented in Section \ref{sec:stab-estim-refeq:16}.
Convergence theorems in the case when $\Phi$ is a nonlocal smoothing operator are discussed in Section \ref{sec:study-conv-case}. Finally, Section~\ref{sec:study-conv-case-1} contains convergence theorems
in the case when $\Phi$ is a local operator.

 \section{Finite difference schemes}
\label{sec:finite-diff-schem}
 Let $N_T$ be a positive integer and $\dt=T/{N_T}$, $t_n= n \dt$, $n=0,\dots, N_T$.
 Let $\T^2_h$ be a uniform grid on the torus with mesh step $h$, (assuming that $1/{h}$ is an integer $N_h$),
 and $x_{ij}$ denote a generic point in  $\T^2_h$.
The values of $u$ and $m$  at $(x_{i,j},t_n)$ are respectively approximated by $u^n_{i,j}$ and $m^n_{i,j}$.
Let $u^n$ (resp. $m^n$) be the vector containing the values $u^n_{i,j}$  (resp. $m^n_{i,j}$), for  $0\le i,j< N_h$ indexed in the lexicographic order.
For all grid function $z$ on $\T_{h}^2$, all  $i$ and $j$, we agree that $z_{i,j}= z_{(i \hbox{ mod } N_h), (j \hbox{ mod }N_h)}$.
\paragraph{Elementary finite difference operators}
Let us introduce the elementary  finite difference operators
  \begin{equation}
\label{eq:8}
 (D_1^+ u )_{i,j} = \frac{ u_{i+1,j}-u_{i,j}   } {h} \quad \hbox{and }\quad  (D_2^+ u )_{i,j} = \frac{ u_{i,j+1}-u_{i,j}   } {h},
\end{equation}
and define $[D_h u]_{i,j}$ as the collection of the four possible one sided finite differences at $x_{i,j}$:
\begin{equation}
  \label{eq:9}
 [D_h u]_{i,j} =\Bigl((D_1^+ u )_{i,j} , (D_1^+ u )_{i-1,j}, (D_2^+ u )_{i,j}, (D_2^+ u )_{i,j-1}\Bigr) \in \R^4.
\end{equation}
We will also need the standard five point discrete Laplace operator
\begin{displaymath}
(\Delta_h u)_{i,j}=  -\frac 1 {h^2} (4u_{i,j} -u_{i+1,j}-u_{i-1,j}-u_{i,j+1}-u_{i,j-1}).
\end{displaymath}
\paragraph{Numerical Hamiltonian}
In order to approximate the term $H(x, \nabla u)$ in (\ref{eq:1}) or (\ref{eq:4}), we consider a  numerical Hamiltonian $g: \T^2 \times \R^4\to \R$,  $(x,q_1,q_2,q_3,q_4)\mapsto g\left(x,q_1,q_2,q_3,q_4\right)$.
Hereafter we will often assume that the following conditions hold:
\begin{description}
\item ($\mathbf{g_1}$)  \emph{monotonicity}: $g$ is nonincreasing with respect to $q_1$ and $q_3$  and nondecreasing with respect to $q_2$ and $q_4$.
\item ($\mathbf{g_2}$)  \emph{consistency:}
  $g\left(x,q_1,q_1,q_2,q_2\right)=H(x,q), \quad \forall x\in \T^2, \forall q=(q_1,q_2)\in \R^2. $
\item ($\mathbf{g_3}$) \emph{differentiability}: $g$ is  of class $\cC^1$.
\item ($\mathbf{g_4}$)  \emph{convexity} : $(q_1,q_2,q_3,q_4)\mapsto g\left(x,q_1,q_2,q_3,q_4\right)$ is convex.
\end{description}
We will approximate $H(\cdot, \nabla u) (x_{i,j})$ by $g(x_{i,j}, [D_h u]_{i,j} )$.\\
 Standard examples of numerical Hamiltonians fulfilling these requirements are provided by  Lax-Friedrichs or Godunov type schemes, see \cite{MR2679575}.
In this work, we  focus on Hamiltonians of the form  $H(x,p)= \ch(x) + |p|^\beta$, for which we choose
\begin{equation}\label{eq:10}
  g(x,q)= \ch(x)+ G(q_1^-,q_2^+, q_3^-, q_4^+),
\end{equation}
where, for a real number $r$, $r^+=\max(r,0)$ and $r^-=\max(-r,0)$ and where
$G:(\R_+)^4\to \R_+$ is given by
\begin{equation}\label{eq:11}
  G(p)= |p|^{\beta}=(p_1^2+p_2^2+p_3^2+p_4^2)^{\frac \beta 2}.
\end{equation}
\paragraph{Discrete version of the cost term  $\Phi[m(t,\cdot)] (x)$}
We introduce the compact and convex set
\begin{equation}
\label{eq:12}
    \cK_h=\{ (m_{i,j})_{ 0\le i,j <N_h}:  h^2 \sum_{i,j} m_{i,j}=1;\quad   m_{i,j}\ge 0 \}
\end{equation}
which can be viewed as the set of the discrete probability measures.
\\
We make the following assumptions, $\Phi_h$ being local or not:
\begin{description}
\item{($\Phi_{h1}$)}
We assume that $\Phi_h$ is continuous on $\cK_h$.
\item{($\Phi_{h2}$)} The numerical cost $\Phi_h$ is monotone in the following sense:
  \begin{equation}
    \label{eq:13}
    \left( \Phi_h[m]-\Phi_h[\tilde m],m-\tilde m \right)_2 \le 0 \Rightarrow \Phi_h[m]= \Phi_h[\tilde m],
  \end{equation}
where $(u,v)_2=\sum_{0\le i,j<N_h} u_{i,j} v_{i,j}$.
This assumption and  ($g_4$) will be  a sufficient condition  for the discrete MFG system to have at most a solution,  $ \Phi_h$ being local or not.
\end{description}
 
If $\Phi$ is a local operator, i.e.  $\Phi[m] (x) = F( m(x))$, $F$ being a continuous function from $\R^+$ to $\R$,
 then the discrete version of $\Phi$ is naturally given by
 $(\Phi_h[m])_{i,j} = F( m_{i,j})$. In this case, the operator $\Phi_h$ is continuous on the set of nonnegative grid functions.
\\
If $\Phi$ is a nonlocal operator, then we  assume that the  discrete operator $\Phi_h$ has the following additional properties:
\begin{description}
\item{($\Phi_{h3}$)}
 We assume that there exists a constant $C$ independent of $h$ such that for all times $t$
and for all grid function  $m\in \cK_h$,
\begin{equation}
  \label{eq:14}
\| \Phi_h[m] \|_{\infty} \le C
\end{equation}
and
  \begin{equation}\label{eq:15}
    |        (\Phi_h[m])_{i,j} - (\Phi_h[m])_{k,\ell} | \le C d_{\T}(x_{i,j},x_{k,\ell})
  \end{equation}
where $d_{\T}(x,y)$ is the distance between the two points $x$ and $y$ in the torus $\T^2$.
\item {($\Phi_{h4}$)}  Define $\cK$ as the set of probability densities, i.e. nonnegative  integrable functions $m$ on $\T^2$ such that $\int_{\T^2} m(x) dx=1$.
For a grid function  $m_h\in \cK_h$, let $\tilde m_h$ be the piecewise bilinear interpolation of $m_h$ at the grid nodes: it is clear that $\tilde m_h\in \cK$.
We assume that there exists a continuous and bounded function
$\omega: \R_+ \to \R_+$ such that $ \omega(0)=0$ and
for all $m \in \cK$, for all sequences $(m_{h})_h$, $m_h\in \cK_h$,
\begin{equation}
  \label{eq:16}
\left \|\, \Phi[m] -\Phi_h[m_h]\, \right \|_{L^\infty (\T_h^2)} \le \omega\left(   \| m-\tilde m_h\|_{L^1(\T^2)} \right).
\end{equation}
 Let $\cI_h m$ be the  grid function whose value at $x_{i,j}$ is
  \begin{displaymath}
     \int_{ | x-x_{i,j}|_{\infty} \le h/2 } m(x) dx.
  \end{displaymath}
It is clear that if $m\in \cK$ then $\cI_h m\in \cK_h$ and that (\ref{eq:16}) implies that
 \begin{equation}
\label{eq:17}
\lim_{h\to 0} \sup_{m\in \cK}\left \|\, \Phi[m] -\Phi_h[\cI_h m]\, \right \|_{L^\infty (\T_h^2)} =0.
\end{equation}
\end{description}
For example, if $\Phi[m]$ is defined as the solution $w$ of the equation $\Delta^2 w + w = m$ in $\T^2$,
($\Delta^2$ being the bilaplacian), then one can define $\Phi_h[m_h]$ as the solution $w_h$ of
$\Delta_h^2 w_h + w_h = m_h$ in $\T_h^2$. It is possible to check that all the above properties are true.
\paragraph{Discrete Bellman equation}
The discrete version of the Bellman equation is obtained by applying a semi-implicit Euler scheme to (\ref{eq:1}),
\begin{equation}\label{eq:18}
\ds   \frac {u^{n+1}_{i,j}- u^{n}_{i,j}} {\dt}  -\nu (\Delta_h u^{n+1})_{i,j} + g(x_{i,j}, [D_h u^{n+1}]_{i,j} )= (\Phi_h[m^{n}])_{i,j},
\end{equation}
for all points in $\T_h^2$ and all $n$, $0\le n < N_T$, where all the discrete operators have been introduced above. Given $(m^n)_{n=0,\dots, N_T-1}$,
(\ref{eq:18}) and the initial condition $u_{i,j}^0= u_0(x_{i,j})$ for all $(i,j)$ completely characterizes  $(u^n)_{0\le n\le N_T}$.
\paragraph{Discrete Fokker-Planck equation}
In order to approximate equation \eqref{eq:2}, it is convenient to  consider its weak formulation which involves in particular the term
\[\ds \int_{\T^2} \diver\left(m \frac {\partial H}{\partial p}(\cdot,\nabla u) \right) (x) w(x)\, dx.\]
By periodicity,
\[\ds - \int_{\T^2} m(x) \frac {\partial H}{\partial p}(x,\nabla u(x)) \cdot \nabla w(x)\, dx\]
holds for any test function $w$. The right hand side in the identity above  will be approximated by
\[ - h^2\sum_{i,j}    m_{i,j} \nabla_q g (x_{i,j},[D_h u]_{i,j}) \cdot [D_h w]_{i,j}
= h^2 \sum_{i,j} \cT_{i,j}(u,m) w_{i,j},\]
where the  transport  operator $\cT$ is defined as follows:
\begin{equation}
\label{eq:19}
  \begin{split}
  &  \cT_{i,j}(u,m)= \\ & \frac 1 h\left(
  \begin{array}[c]{l}
\ds
\left(
  \begin{array}[c]{l}
   \ds  m_{i,j}  \frac {\partial g} {\partial q_1} (x_{i,j}, [D_h u]_{i,j})
     - m_{i-1,j}  \frac {\partial g} {\partial q_1}(x_{i-1,j}, [D_h u]_{i-1,j}) \\ \ds
     +m_{i+1,j}  \frac {\partial g} {\partial q_2} (x_{i+1,j},[D_h u]_{i+1,j})
-   \ds  m_{i,j}  \frac {\partial g} {\partial q_2} (x_{i,j}, [D_h u]_{i,j})
  \end{array}
\right)
\\ \\
\ds +
\left(
  \begin{array}[c]{l}
\ds  m_{i,j}  \frac {\partial g} {\partial q_3} (x_{i,j}, [D_h u]_{i,j})
-  \ds  m_{i,j-1}\frac {\partial g} {\partial q_3} (x_{i,j-1}, [D_h u]_{i,j-1}) \\
+ \ds  m_{i,j+1}\frac {\partial g} {\partial q_4} (x_{i,j+1}, [D_h u]_{i,j+1})
-  m_{i,j}  \frac {\partial g} {\partial q_4} (x_{i,j}, [D_h u]_{i,j})
  \end{array}
\right)
  \end{array}
\right).
  \end{split}
\end{equation}
The discrete version of equation (\ref{eq:2}) is chosen as follows:
\begin{equation}
\label{eq:20}
\ds   \frac {m^{n+1}_{i,j}- m^{n}_{i,j}} {\dt} +\nu (\Delta_h m^{n})_{i,j}
+\cT_{i,j}(u^{n+1},m^n)= 0.
\end{equation}
This scheme is implicit w.r.t. to $m$ and explicit w.r.t. $u$ because the considered Fokker-Planck equation is backward.
Given $u$ this is a system of linear equations for $m$.
It is easy to see that if $m^n$ satisfies (\ref{eq:20}) for $0\le n<N_T$ and if $m^{N_T}\in \cK_h$, then $m^n\in \cK_h$ for all $n$,  $0\le n<N_T$.
\begin{remark}\label{sec:numer-schem-refeq:7-2}
It is important to realize that the
operator $m\mapsto  \bigl(-\nu (\Delta_h m)_{i,j} - \cT_{i,j}(u,m)\bigr)_{i,j}$
 is the adjoint of the linearized version of the operator
 $u\mapsto   \bigl( -\nu (\Delta_h u)_{i,j} + g(x_{i,j}, [D_h u]_{i,j} )\bigr)_{i,j}$.
\end{remark}
\paragraph{Summary}
The fully discrete scheme
 for  system (\ref{eq:1}),(\ref{eq:2}),(\ref{eq:3})
is therefore the following: for all $0\le i,j< N_h$ and $ 0\le n < N_T$
\begin{equation}
\label{eq:21}
\left\{
  \begin{array}[c]{llr}
\frac {u^{n+1}_{i,j}- u^{n}_{i,j}} {\dt}  -\nu (\Delta_h u^{n+1})_{i,j}
 + g(x_{i,j}, [D_h u^{n+1}]_{i,j} )&=(\Phi_h[m^{n}])_{i,j}), \\
\frac {m^{n+1}_{i,j}- m^{n}_{i,j}} {\dt} +\nu (\Delta_h m^{n})_{i,j}
+\cT_{i,j}(u^{n+1},m^n)&=0,
\end{array}\right.
\end{equation}
with the initial and terminal conditions
\begin{equation}
  \label{eq:22}
m_{i,j}^{N_T}=  \frac 1 {h^2} \int_{|x-x_{i,j}|_{\infty}\le h/2} m_T(x) dx,\quad\quad  u_{i,j}^0=u_0(x_{i,j}),\quad 0\le i,j< N_h.\end{equation}
The following theorem was proved in \cite{MR2679575} (using essentially a Brouwer fixed point argument and estimates on the solutions of the discrete Bellman equation):
\begin{theorem}
  \label{sec:exist-discr-probl-1}
Assume  that $\nu>0$, that ($g_1$)--($g_3$) and ($\Phi_{h1}$) hold, that $u_0$ is a continuous function on $\T^2$ and that $m_T\in \cK$: then (\ref{eq:21})--(\ref{eq:22}) has a solution such that $m^n\in \cK_h$,  $\forall n$.
\\
If furthermore
\begin{itemize}
\item  ($\Phi_{h3}$) holds
\item  there exists a constant $C$ such that \vspace{5pt}
\begin{displaymath}
\left|\frac{\partial g}{\partial x} (x, (q_1,q_2,q_3,q_4))\right|\le C(1+|q_1|+|q_2|+|q_3|+|q_4|)\quad  \forall x\in \T^2,\, \forall q_1,q_2,q_3,q_4
\end{displaymath}
\item $u_0$ is Lipschitz continuous
\end{itemize}
then  $\max_{0\le n \le N_T} \bigl(\|u^n\|_{\infty} + \|D_h u^n \|_{\infty} \bigr) \le c$ for a constant $c$ independent of $h$ and $\dt$.
\end{theorem}
\begin{remark}
  Note that the technical condition on $\frac{\partial g}{\partial x}$ is automatically true if $g$ is given by  (\ref{eq:10})-(\ref{eq:11}) and $h$ is $\cC^1$.
\end{remark}
\begin{remark}
  A priori estimates in the case when $\Phi$ is a local operator will be given in \S~\ref{sec:stab-estim-refeq:16}.
\end{remark}
Since (\ref{eq:21})-(\ref{eq:22}) has exactly the same structure as the continuous problem (\ref{eq:1})-(\ref{eq:3}), uniqueness has been obtained in \cite{MR2679575} with the same arguments as in \cite{MR2271747}:
\begin{theorem}
  Assume  that $\nu>0$, that ($g_1$)--($g_4$) and ($\Phi_{h1}$)-($\Phi_{h2}$) hold,
then  (\ref{eq:21})--(\ref{eq:22}) has a unique solution.
\end{theorem}

\section{Basic facts for numerical Hamiltonians of the form (\ref{eq:10})-(\ref{eq:11})}
\label{sec:basic-facts-numer-1}
We focus on numerical Hamitonians in the form (\ref{eq:10})-(\ref{eq:11}) but obviously, what follows
 holds for  $g(x,q)= \ch(x) + c G(q)$, where $G$ is given by (\ref{eq:11}) and $c$ is a positive constant.\\
We use the following notations: for $p\in \R^4$, $|p|$ (resp. $|p|_\infty$) is the Euclidean norm of $p$ (resp. the max norm: $|p|_\infty= \max_{i=1,\dots,4} |p_i|$).
For a function  $p \to \psi(p)\in \R$,   $p\in \R^4$,  $\psi_p(p)\in \R^4$ (resp.  $\psi_{pp}(p)\in \R^{4\times 4}$)  will stand for the gradient of $\psi$  (resp.  the Hessian of $\psi$).
For a function $(x,p) \to \psi(x,p)\in \R$,  $x\in \T_2$, $p\in \R^4$,  $\psi_p(x,p)\in \R^4$ is the gradient of $p \to \psi(x,p)$  and  $\psi_{pp}(x,p)\in \R^{4\times 4}$ is the  Hessian of $p\to \psi(x,p)$.
\subsection{Basic lemmas}\label{sec:basic-lemmas}
Let us state a few lemmas about $g$ and $G$:
\begin{lemma}
  \label{sec:case-when-hx}
For all $p\in (\R_+)^4$, $G_p(p)=\beta |p|^{\beta-2} p$ and
  \begin{equation}\label{eq:23}
    G_{pp}(p)= \beta |p|^{\beta-2}I_4 + \beta (\beta-2)  |p|^{\beta-4}  p \otimes p.
  \end{equation}
If $\beta\ge 2$, then   $G_{pp}(p)- \beta |p|^{\beta-2}I_4$ is a positive semi-definite  matrix, which may be written
\begin{equation}
  \label{eq:24}
G_{pp}(p)\ge \beta |p|^{\beta-2}I_4.
\end{equation}
If $1<\beta<2$, then
\begin{equation}
  \label{eq:25}
G_{pp}(p) \ge \beta(\beta -1) |p|^{\beta-2}I_4.
\end{equation}
\end{lemma}

\begin{lemma}
  \label{sec:case-when-hx-1}
For all $q, \tilde q\in \R^4$, let $ p,\tilde p \in (\R_+)^4$ be given by
\begin{equation}
  \label{eq:26}
  p=(q_1^-,q_2^+, q_3^-, q_4^+),\quad \hbox{and}  \quad\tilde p=(\tilde q_1^-,\tilde q_2^+, \tilde q_3^-, \tilde q_4^+).
\end{equation}
We have that
\begin{equation}
  \label{eq:27}
g(x,\tilde q)-g(x, q) -g_q(x, q)\cdot (\tilde q -q )\ge G(\tilde p)-G( p) -G_p(p)\cdot (\tilde p -p ).
\end{equation}
\end{lemma}
\begin{proof}
See Appendix \ref{sec:proofs-some-techn}.
\end{proof}
\begin{lemma}
  \label{sec:case-when-hx-2}
If  $\beta\ge 2$, then for all $q, \tilde q\in \R^4$, for  $ p,\tilde p \in (\R_+)^4$  given by~(\ref{eq:26})
\begin{eqnarray}
  \label{eq:29}
g(x,\tilde q)-g(x, q) -g_q(x, q)\cdot (\tilde q -q )& \ge & \frac 1 {\beta-1}  \max(|p|^{\beta-2},|\tilde p|^{\beta-2}) |p-\tilde p|^2
\\  \label{eq:30} &\ge& \frac { |p-\tilde p|^\beta} { 2^{\beta-2} (\beta-1)}   .
\end{eqnarray}
If  $1< \beta< 2$, then for all $q, \tilde q\in \R^4$ such that for $ p,\tilde p $ given by~(\ref{eq:26}),  $\tilde p +p \not =0$, we have
\begin{equation}
\label{eq:31}
g(x,\tilde q)-g(x, q) -g_q(x, q)\cdot (\tilde q -q )\ge    2^{\beta-3} \beta (\beta-1)  \min( |p|_\infty ^{\beta-2},|\tilde p|_\infty^{\beta-2}) |p-\tilde p|^2.
\end{equation}
\end{lemma}
\begin{proof}
See Appendix \ref{sec:proofs-some-techn}.
\end{proof}

\begin{lemma}
  \label{sec:case-when-hx-3}
For $\beta\ge 2$, there exists a positive constant $c$ such that
for all $q, \tilde q, r \in \R^4$, for all $\eta>0$,
\begin{equation}
\label{eq:32}
\left|\left(g_q(x, \tilde q) -g_q(x, q)\right)\cdot r \right| \le         \max(|p|^{\beta-2},|\tilde p|^{\beta-2}) \left(\frac c \eta |p-\tilde p|^2     +  \eta |r|^2  \right) .
\end{equation}
where  $ p,\tilde p \in (\R_+)^4$ are given by~(\ref{eq:26}).
\end{lemma}
\begin{proof}
See Appendix \ref{sec:proofs-some-techn}.
\end{proof}
\subsection{A nonlinear functional $\cG(m,u,\tilde u)$}
\label{sec:nonl-oper-cgm}
Let us define the nonlinear functional $\cG$ acting on grid functions by
\begin{equation}
  \label{eq:34}
  \begin{split}
&\cG(m, u, \tilde u) \\=
&\sum_{n=1}^{N_T}     \sum_{i,j}   m^{n-1}_{i,j} \Bigl(g(x_{i,j}, [D\tilde u^{n} ]_{i,j})-g(x_{i,j},  [Du^{n} ]_{i,j}) -g_q(x_{i,j},  [Du^{n} ]_{i,j}) \cdot ( [D\tilde u^{n} ]_{i,j} - [D u^{n} ]_{i,j})  \Bigr).
  \end{split}
\end{equation}
Under Assumption ($g_4$), it is clear that $\cG(m, u, \tilde u)\ge 0$ if $m$ is a nonnegative grid function. If $g$ is of the form (\ref{eq:10})-(\ref{eq:11}), we have a more precise estimate:
\begin{lemma}
  \label{sec:basic-facts-numer}
 If $m$ is a nonnegative grid function and if  $g$ is of the form (\ref{eq:10})-(\ref{eq:11}) with $\beta\ge 2$, then
 \begin{equation}
\label{eq:35}
\begin{split}
\cG(m, u, \tilde u) &
\ge \frac 1 {\beta-1}  \sum_{n=1}^{N_T}     \sum_{i,j}  m_{i,j}^{n-1} \max(|p_{i,j}^n|^{\beta-2},|\tilde p_{i,j}^n|^{\beta-2}) |p_{i,j}^n-\tilde p_{i,j}^n|^2\\
&\ge \frac 1 { 2^{\beta-2} (\beta-1)}    \sum_{n=1}^{N_T}     \sum_{i,j}  m_{i,j}^{n-1} |p_{i,j}^n-\tilde p_{i,j}^n|^\beta.
\end{split}
 \end{equation}
where
 $p^n_{i,j}$ and $\tilde p^n_{i,j}$ are the four dimensional vectors
 \begin{equation}
   \label{eq:36}
  \begin{split}
  p_{i,j}^n= \left(    \Bigl(  \left(D_1^+ u^n )_{i,j}\right)^- , \left((D_1^+ u^n )_{i-1,j}\right)^+, \left((D_2^+ u^n )_{i,j}\right)^-, \left((D_2^+ u^n )_{i,j-1}\right)^+\Bigr)  \right),\\
    \tilde p_{i,j}^n= \left(    \Bigl(  \left(D_1^+ \tilde u^n )_{i,j}\right)^- , \left((D_1^+ \tilde u^n )_{i-1,j}\right)^+, \left((D_2^+ \tilde u^n )_{i,j}\right)^-, \left((D_2^+ \tilde u^n )_{i,j-1}\right)^+\Bigr)  \right).
  \end{split}
 \end{equation}
Furthermore, if $m$ is bounded from below by $\underline m$, then
 \begin{equation}
   \label{eq:37}
\cG(m, u, \tilde u)  \ge   \frac  {\underline m} { 2^{2\beta-3}  (\beta-1)}    \sum_{n=1}^{N_T}    \sum_{i,j}     \Bigl| [D\tilde u^n ]_{i,j} - [D u^n ]_{i,j}\Bigr|^\beta.
 \end{equation}
\end{lemma}
\begin{proof}
We deduce (\ref{eq:35}) from  (\ref{eq:29}) and (\ref{eq:30}).
If $m$ is bounded from below by $\underline m>0$, we deduce  that
\begin{displaymath}
  \begin{split}
    \cG(m, u, \tilde u) & \ge \frac  {\underline m} { 2^{\beta-2} (\beta-1)}  \sum_{n=1}^{N_T}     \sum_{i,j}  |p_{i,j}^n-\tilde p_{i,j}^n|^\beta \\
& \ge \frac  {\underline m} { 2^{\beta-2} (\beta-1)}  \sum_{n=1}^{N_T}     \sum_{i,j}      \sum_{k=1}^4 | (p_{i,j}^n)_k-(\tilde p_{i,j}^n)_k |^\beta.
  \end{split}
\end{displaymath}
Since for each $0\le i,j <N_h$, each $1\le k\le 4$, the quantity $ \Bigl| ([D\tilde u^n ]_{i,j})_k - ([D u^n ]_{i,j})_k \Bigr|^\beta$ appears at least once in the sum above, we deduce that
\begin{displaymath}
  \begin{split}
    \cG(m, u, \tilde u) & \ge  \frac  {\underline m} { 2^{\beta-1} (\beta-1)}  \sum_{n=1}^{N_T}    \sum_{i,j}   \sum_{k=1}^4    \Bigl| ([D\tilde u^n ]_{i,j})_k - ([D u^n ]_{i,j})_k \Bigr|^\beta \\
& \ge   \frac  {\underline m} { 2^{2\beta-3}  (\beta-1)}    \sum_{n=1}^{N_T}    \sum_{i,j}     \Bigl| [D\tilde u^n ]_{i,j} - [D u^n ]_{i,j}\Bigr|^\beta
  \end{split}
\end{displaymath}
where the second inequality comes from the fact that for four positive numbers $(a_k)_{k=1,\dots, 4}$ $\sum_{k=1}^4 a_k^\beta \ge \frac 1 {4^{\beta/2 -1}} (\sum_{k=1}^4 a_k^2)^{\beta/2}$.
\end{proof}
\begin{remark}\label{sec:basic-identity-}
  The same kind of argument shows that if $1<\beta<2$, and $m_{i,j}^n\ge \underline m$, then
  \begin{displaymath}
    \begin{split}
 \cG(m, 0, u) &\ge 2^{\beta-3} \beta (\beta-1)   \underline m   \sum_{n=1}^{N_T}     \sum_{i,j}  |p_{i,j}^n|_\infty ^{\beta-2}  |p_{i,j}^n|^2
 \ge 2^{\beta-3} \beta (\beta-1)   \underline m   \sum_{n=1}^{N_T}     \sum_{i,j}   |p_{i,j}^n|^\beta \\
&\ge        2^{ 2\beta-5} \beta (\beta-1)   \underline m   \sum_{n=1}^{N_T}     \sum_{i,j} \sum_{k=1}^4   |(p_{i,j}^n)_k|^\beta
\ge        2^{ 2\beta-6} \beta (\beta-1)  \underline m    \sum_{n=1}^{N_T}     \sum_{i,j} \sum_{k=1}^4   |( [D u^n ]_{i,j})_k|^\beta\\
& \ge        2^{ 2\beta-6} \beta (\beta-1)   \underline m   \sum_{n=1}^{N_T}     \sum_{i,j}   | [D u^n ]_{i,j}|^\beta,
    \end{split}
  \end{displaymath}
where $p_{i,j}^n$ is given by (\ref{eq:36}).
\end{remark}
\subsection{A fundamental identity }\label{sec:stab-refeq:19:-basic}
In this paragraph,  we discuss a key identity which leads to the stability of the finite difference scheme under additional assumptions.
Consider a perturbed system:
  \begin{equation}
\label{eq:38}
\left\{
  \begin{array}[c]{llr}
\frac {\tilde u^{n+1}_{i,j}- \tilde u^{n}_{i,j}} {\dt}  -\nu (\Delta_h \tilde u^{n+1})_{i,j}
 + g(x_{i,j}, [D_h \tilde u^{n+1}]_{i,j} )&=(\Phi_h[\tilde m^n])_{i,j} + a^n_{i,j}  , \\
\frac {\tilde m^{n+1}_{i,j}- \tilde m^{n}_{i,j}} {\dt} +\nu (\Delta_h \tilde m^{n})_{i,j}
+\cT_{i,j}(\tilde u^{n+1},\tilde m^n)&= b^n_{i,j}.
\end{array}\right.
\end{equation}
Multiplying the first equations in (\ref{eq:38}) and (\ref{eq:21})   by $ m_{i,j}^n-\tilde  m_{i,j}^n$  and subtracting, then summing the results
 for all $n=0,\dots, N_T-1$ and all $(i,j)$, we obtain
\begin{equation}
\label{eq:39}
\begin{split}
   & \sum_{n=0}^{N_T-1} \frac 1 {\dt} {((u^{n+1}- \tilde u^{n+1}) -(u^{n}- \tilde u^{n}), (m^n-\tilde m^n))_2 }   -\nu (\Delta_h (u^{n+1} -\tilde u^{n+1}), m^n-\tilde m^n  )_2 \\
&
 + \sum_{n=0}^{N_T-1} \sum_{i,j} (g(x_{i,j}, [D_h u^{n+1}]_{i,j} )- g(x_{i,j}, [D_h \tilde u^{n+1}]_{i,j} ))  (m^n_{i,j}-\tilde m^n_{i,j})
\\= &  \sum_{n=0}^{N_T-1} \left(\Phi_h[m^{n}]-\Phi_h[\tilde m^{n}], m^{n}-\tilde m^{n}  \right)_2  -\sum_{n=0}^{N_T-1} ( a^n, m^n-\tilde m^n )_2,
\end{split}
\end{equation}
where $(X,Y)_2=\sum_{i,j}X_{i,j} Y_{i,j}$. Similarly,  subtracting the second equation in (\ref{eq:38}) from  the second equation in (\ref{eq:20}),
 multiplying the result by $ u_{i,j}^{n+1}-\tilde u_{i,j}^{n+1}$ and  summing for all $n=0,\dots, N_T-1$ and all $(i,j)$ leads to
\begin{equation}
\label{eq:40}
  \begin{split}
&    \sum_{n=0}^{N_T-1}   \frac 1 \dt ( (m^{n+1}- m^{n}) - (\tilde m^{n+1}- \tilde m^{n}),  (u^{n+1}-\tilde u^{n+1}))_2 +\nu ((m^n- \tilde m^n),\Delta_h (u^{n+1}-\tilde u^{n+1}))_2 \\
&
 -
     \sum_{n=0}^{N_T-1}  \sum_{i,j}   m^n_{i,j} {[D_h (u^{n+1}-\tilde u^{n+1})]}_{i,j}\cdot  g_q \left(x_{i,j}, [D_h u^{n+1}]_{i,j}\right)\\
& \ds
+
      \sum_{n=0}^{N_T-1} \sum_{i,j}   \tilde m^n_{i,j} {[D_h (u^{n+1}-\tilde u^{n+1})]}_{i,j}\cdot  g_g \left(x_{i,j}, [D_h \tilde u^{n+1}]_{i,j}\right)=
 -\sum_{n=1}^{N_T} ( b^{n-1}, u^n-\tilde u^n )_2.
  \end{split}
\end{equation}
Adding (\ref{eq:39}) and (\ref{eq:40})  leads to the fundamental identity
\begin{equation}
  \label{eq:41}
  \begin{split}
 &- \frac 1 \dt  (m^{N_T} -\tilde m^{N_T}, u^{N_T} -\tilde u^{N_T} )_2 +  \frac 1 \dt  (m^{0} -\tilde m^{0}, u^{0} -\tilde u^{0} )_2 \\
&+\cG(m, u,\tilde u) + \cG(\tilde m ,\tilde u,u) +  \sum_{n=0}^{N_T-1}   (\Phi_h[m^n] -\Phi_h[\tilde m^n], m^n-\tilde m^n )_2 \\
= & \sum_{n=0}^{N_T-1}  ( a^n, m^n-\tilde m^n )_2   +\sum_{n=1}^{N_T} ( b^{n-1}, u^n-\tilde u^n )_2.
  \end{split}
\end{equation}
It is important to note that under assumptions ($g_4$) and ($\Phi_{h2})$, the second line of (\ref{eq:41}) is made of three nonnegative terms.
 This is the key observation leading to uniqueness for (\ref{eq:21})-(\ref{eq:22}), but it may also lead to a priori estimates or stability estimates under
 additional assumptions including for example the assumptions of Lemma \ref{sec:basic-facts-numer}.

\section{Study of the convergence in the case when $\Phi$ is a nonlocal smoothing operator}
\label{sec:study-conv-case}
Hereafter the Hamiltonian is of the form  (\ref{eq:7}).
In all the following convergence results, we assume that  system (\ref{eq:1})-(\ref{eq:3}) has a unique classical solution.
 This is always the case
if $\Phi$  is monotonous in the sense of (\ref{eq:6})
and  continuously maps  the set of probability  measures on $\T^2$ (endowed with the weak * topology) to a bounded subset of $Lip(\T^2)$, the Lipschitz functions on $\T^2$,  and for example maps continously $C^{k,\alpha}(\T^2)$ to  $C^{k+1,\alpha}(\T^2)$, for all $k\in \N$ and $0\le \alpha <1$.
\\
We summarize the assumptions made in \S~\ref{sec:study-conv-case} as follows:
\paragraph{Standing assumptions (in Section \ref{sec:study-conv-case})}
\begin{itemize}
\item We take $\nu>0$
\item  The Hamiltonian is of the form  (\ref{eq:7}) and the function $x\to \ch(x)$ is $\cC^1$ on $\T^2$
\item  the functions $u^0$ and $m_T$ are smooth, and  $m_T\in \cK$ is bounded from below by a positive number
\item  We assume that  $\Phi$ is monotone in the sense of (\ref{eq:6}), nonlocal and smoothing,
so  that there is a unique classical solution $(u,m)$ of (\ref{eq:1})-(\ref{eq:3}) such that $m>0$
\item We consider a numerical Hamiltonian given by (\ref{eq:10})-(\ref{eq:11})
and a numerical cost function  $\Phi_h$ such that ($\Phi_{h1}$),  ($\Phi_{h2}$),  ($\Phi_{h3}$), and ($\Phi_{h4}$) hold.
\end{itemize}

\subsection{The case when $\beta\ge 2$}
\label{sec:case-when-betage}

\begin{theorem}\label{sec:case-when-hx-4}
We make the standing assumptions stated at the beginning of \S~\ref{sec:study-conv-case} and we choose $\beta\ge 2$.
\\
Let $u_h$ (resp. $m_h$) be the   piecewise trilinear function  in $\cC([0,T]\times\T^2)$ obtained by interpolating the values $u_{i,j}^n$ (resp $m_{i,j}^n$)  at the nodes of the space-time grid.
The functions $u_h$ converge uniformly and  in $ L^\beta(0,T; W^{1,\beta}(\T^2))$ to $u$ as $h$ and $\dt$ tend to $0$. The functions $m_h$ converge to $m$ in $C^0([0,T]; L^2(\T^2))\cap L^2(0,T; H^1(\T^2))$ as $h$ and $\dt$ tend to $0$.
 \end{theorem}
\begin{proof}
  Note that $m_h(t,\cdot) \in \cK$ for any $t\in [0,T]$.\\
We call $\wtU^n$ and $\wtM^n$  the grid functions such that $\wtU^n_{i,j}=   u(n\dt,x_{i,j}) $ and $\wtM^n=\cI_h(m(t_n,\cdot))$.
The functions $\wtU^n$ and $\wtM^n$ are solutions of (\ref{eq:38}) where $a$ and $b$ are consistency errors.
From the fact that $(u,m)$ is a classical solution of (\ref{eq:1})-(\ref{eq:3}), we infer from the consistency of the scheme (in particular from (\ref{eq:17}))
that $\max_{0\le n < N_T} (\|a^n\|_{L^\infty(\T_h^2)}+\|b^n\|_{L^\infty(\T_h^2)}) $ tends to zero as $h$ and $\dt$ tend to zero.
\paragraph{Step 1}
As a consequence of the previous observations, the fundamental identity (\ref{eq:41}) holds, and from (\ref{eq:22}), can be written as follows:
\begin{equation}
\label{eq:45}
  \begin{split}
& h^2 \dt \cG(m, u,\tilde u) +  h^2 \dt \cG(\tilde m ,\tilde u,u) +  h^2 \dt  \sum_{n=0}^{N_T-1}   (\Phi_h[m^n] -\Phi_h[\tilde m^n], m^n-\tilde m^n )_2 \\
= &  h^2 \dt\sum_{n=0}^{N_T-1}  ( a^n, m^n-\tilde m^n )_2   + h^2 \dt \sum_{n=1}^{N_T} ( b^{n-1}, u^n-\tilde u^n )_2.
  \end{split}
\end{equation}
The a priori estimate on the discrete Lipschitz norm of $u_h$ given at the end of Theorem \ref{sec:exist-discr-probl-1} implies that
$ \lim_{h, \dt\to 0}   h^2  \max_n |( b^{n-1}, u^n-\tilde u^n )_2| =0$. From the fact that $m^n\in \cK_h$, we also get that $  \lim_{h, \dt\to 0}   h^2  \max_n| ( a^n, m^n-\tilde m^n )_2|=0$.
We can then use (\ref{eq:35}) to deduce that
\begin{equation}
  \label{eq:46}
  \begin{array}[c]{rcl}
\ds \dt \, h^2 \sum_{n=0}^{N_T-1}  \sum_{i,j}   \wtM^n_{i,j}  \max\left(|p ^{n+1}_{i,j}|^{\beta-2}, |\tilde p ^{n+1}_{i,j}|^{\beta-2} \right) |p ^{n+1}_{i,j}-\tilde p ^{n+1}_{i,j}|^2 &=&o(1),\\
\ds \dt \, h^2 \sum_{n=0}^{N_T-1}  \sum_{i,j}   m^n_{i,j}  \max\left(|p ^{n+1}_{i,j}|^{\beta-2}, |\tilde p ^{n+1}_{i,j}|^{\beta-2} \right)  |p ^{n+1}_{i,j}-\tilde p ^{n+1}_{i,j}|^2 &=& o(1).
  \end{array}
\end{equation}
where $p ^{n}_{i,j}$ and $\tilde p ^{n}_{i,j}$ are given by (\ref{eq:36}). Since $\tilde m^n_{i,j}$ is bounded from below by a positive constant, we also deduce from (\ref{eq:37})  that
\begin{equation}
  \label{eq:47}
h^2 \dt \sum_{n=1}^{N_T}    \sum_{i,j}     \Bigl| [D_h\tilde u^n ]_{i,j} - [D_h u^n ]_{i,j}\Bigr|^\beta =o(1).
\end{equation}
\paragraph{Step 2} We introduce $e^\ell=m^\ell- \wtM^\ell$.  Subtracting the second equation in (\ref{eq:38})  from  the second equation in (\ref{eq:21}),
  multiplying the result by $ e_{i,j}^{\ell}$ and summing for all $\ell=0,\dots, n-1$ and all $(i,j)$ leads to
 \begin{displaymath}
     \begin{split}
&
 h^2 \sum_{\ell=0}^{n-1} (  e^{\ell+1}-e^\ell, e^{\ell})_2+ \dt \, h^2 \sum_{\ell=0}^{n-1}  \nu (\Delta_h e^\ell, e^\ell)_2
\\
  &- \dt \, h^2   \sum_{\ell=0}^{n-1}  \sum_{i,j}   {[D_h  e^\ell ]}_{i,j}\cdot \left( m^\ell_{i,j}  g_q (x_{i,j}, [D_hu^{\ell+1}] _{i,j}) -  \wtM^\ell_{i,j}  g_q (x_{i,j},   [D_h\tilde u^{\ell+1}] ^{\ell+1}_{i,j})\right)
\\
=& - h^2 \dt \sum_{\ell=0}^{n-1} ( b^{\ell},  e^\ell )_2.
  \end{split}
\end{displaymath}
This implies that
\begin{equation}
  \label{eq:48}
     \begin{split}
&   \dt \, h^2 \sum_{\ell=0}^{n-1}  \left( \frac 1 \dt (  e^{\ell+1}-e^\ell, e^{\ell})_2 +\nu (\Delta_h e^\ell, e^\ell)_2   -  \sum_{i,j} e^\ell_{i,j}  {[D_h  e^\ell ]}_{i,j}\cdot  g_q (x_{i,j}, [D_h u^{\ell+1}] _{i,j} ) \right)
\\
=& - h^2 \dt \sum_{\ell=0}^{n-1} ( b^{\ell},  e^\ell )_2   -  \dt \, h^2   \sum_{\ell=0}^{n-1}  \sum_{i,j}  \wtM^\ell_{i,j}  {[D_h  e^\ell ]}_{i,j}\cdot   (g_q (x_{i,j}, [D_h \tilde u^{\ell+1}] _{i,j}) - g_q (x_{i,j},  [D_h u^{\ell+1}] _{i,j}))   .
  \end{split}
\end{equation}
It is clear that \[  h^2 \dt \left|\sum_{\ell=0}^{n-1} ( b^{\ell},  e^\ell )_2 \right|= o(1) \left( h^2  \dt \sum_{\ell=0}^{n-1} \|e^\ell\|^2_{2}\right)^{\frac 1 2} .\]
From Lemma \ref{sec:case-when-hx-3}, we know that there exists an absolute constant $c$ such that for all $\eta>0$,
\begin{equation*}
  \begin{array}[c]{ll}
&\ds \left|\wtM^\ell_{i,j}  {[D_h  e^\ell ]}_{i,j}\cdot   (g_q (x_{i,j},  [D_h \tilde u^{\ell+1}] _{i,j}) - g_q (x_{i,j},     [D_h u^{\ell+1}] _{i,j}) \right| \\ \le &\ds
        \wtM^\ell_{i,j} \max(|p^{\ell+1}_{i,j}|^{\beta-2},|\tilde p^{\ell+1}_{i,j}|^{\beta-2}) \left(\frac c \eta |p^{\ell+1}_{i,j}-\tilde p^{\ell+1}_{i,j}|^2     +  \eta |{[D_h  e^\ell ]}_{i,j}|^2  \right)
  \end{array}
\end{equation*}
where  $ p^{\ell+1}_{i,j},\tilde p^{\ell+1}_{i,j} \in (\R_+)^4$ are given by~(\ref{eq:36}). Using the  $L^\infty$ bound on $ [D_h  u^{\ell+1}] $ uniform w.r.t. $\ell$, $h$ and $\dt$ given in Theorem~\ref{sec:exist-discr-probl-1},
 we obtain  that there exists a constant $c$ such that for all $\eta>0$,
\begin{equation}
  \label{eq:49}
  \begin{split}
    &\ds \left|\wtM^\ell_{i,j}  {[D_h  e^\ell ]}_{i,j}\cdot   (g_q (x_{i,j},  [D_h \tilde u^{\ell+1}] _{i,j}) - g_q (x_{i,j},  [D_h u^{\ell+1}] _{i,j}) \right| \\ \le &\ds
 \wtM^\ell_{i,j}  \left(\frac c \eta \left|p^{\ell+1}_{i,j}-\tilde p^{\ell+1}_{i,j}\right|^\beta     +  \eta |{[D_h  e^\ell ]}_{i,j}|^2  \right) \\
\le &\ds
 \wtM^\ell_{i,j}  \left(\frac c \eta \left| [D_h u^{\ell+1}]_{i,j}- [D_h \tilde u^{\ell+1}]_{i,j}\right |^\beta     +  \eta |{[D_h  e^\ell ]}_{i,j}|^2  \right) .
  \end{split}
\end{equation}
We shall also use the standard estimate :
\begin{equation}
  \label{eq:50}
 \left \| [D_h  e^\ell] \right \|_2^2 \le- C (\Delta_h e^\ell, e^\ell)_2,
\end{equation}
for a positive  constant $C$ independent of $h$. Finally, from (\ref{eq:48}),   a very classical argument making use of (\ref{eq:49}),~(\ref{eq:50}) and (\ref{eq:47}),
the  $L^\infty$ bound on $ [D_h  u^{\ell+1}] $ uniform w.r.t. $\ell$,
 $h$ and $\dt$, and the fact that $h\|e^{N_T}\|_2=o(1)$, leads to the estimate:
\begin{equation}\label{eq:51}
  \max_{0\le \ell< N_T} h^2 \| e^\ell\|_2^2  +\dt\,h^2    \sum_{\ell=0}^{N_T-1} \left \| [D_h e^\ell] \right \|_2^2  = o(1).
\end{equation}
We easily deduce from (\ref{eq:51}) the claim on the convergence of $m_h$ to $m$.
\paragraph{Step 3}
We have found that $ \max_{0\le \ell< N_T} h^2 \| m^\ell -\cI_h(m (t_\ell, \cdot))  \|_2^2 =o(1)$. From ($\Phi_{h4}$), 
 this implies that
\begin{displaymath}
  \frac {u^{n+1}_{i,j}- u^{n}_{i,j}} {\dt}  -\nu (\Delta_h u^{n+1})_{i,j}
 + g(x_{i,j}, [D_h u^{n+1}]_{i,j} )=(\Phi[m(t_{n+1},\cdot)])(x_{i,j})  +o(1),
\end{displaymath}
The uniform convergence of the piecewise linear functions $u_h$
(defined by interpolating the values $u_{i,j}^n$) to $u$ is obtained from classical results on the approximation of Bellman equations by consistent and monotone schemes.
 From this and (\ref{eq:47}), we also deduce the convergence of $u_h$ to $u$ in $ L^\beta(0,T; W^{1,\beta}(\T^2))$.
\end{proof}

\begin{remark}
In the present case, additional assumptions on the order of the discrete scheme should lead to error estimates. We will not discuss this matter.
\end{remark}

\subsection{The case when $1<\beta< 2$}
\label{sec:case-when-1beta-1}

\begin{theorem}\label{sec:case-when-1beta}
We make the standing assumptions stated at the beginning of \S~\ref{sec:study-conv-case} and we choose $\beta$ such  $1<\beta<2$.
\\
Let $u_h$ (resp. $m_h$) be the   piecewise trilinear function  in $\cC([0,T]\times\T^2)$ obtained by interpolating the values $u_{i,j}^n$ (resp $m_{i,j}^n$)  at the nodes of the space-time grid.
The functions $u_h$ converge uniformly and   in $ L^2(0,T; W^{1,2}(\T^2))$ to $u$ as $h$ and $\dt$ tend to $0$. The functions $m_h$ converge to $m$
in $L^2((0,T)\times \T^2)$ as $h$ and $\dt$ tend to $0$.
 \end{theorem}
\begin{proof}
  \paragraph{Step 1} We start from (\ref{eq:45}) where $a$ and $b$ are the same consistency errors (with the same bounds) as in the previous paragraph;  using (\ref{eq:31}), this implies that
  \begin{equation}
    \label{eq:52}
      2^{\beta -3} \beta(\beta-1)        \dt \, h^2 \sum_{\ell=0}^{N_T-1}  \sum_{i,j}   \wtM^\ell_{i,j}     1_{ \{p ^{\ell+1}_{i,j} +\tilde p ^{\ell+1}_{i,j}\not=0\} }
 \min \left(| p ^{\ell+1}_{i,j}|_{\infty}^{\beta-2},| \tilde p ^{\ell+1}_{i,j}|_{\infty}^{\beta-2}  \right)
 |p ^{\ell+1}_{i,j}-\tilde p ^{\ell+1}_{i,j}|^2 =o(1).
  \end{equation}
where  $ p^{\ell+1}_{i,j},\tilde p^{\ell+1}_{i,j} \in (\R_+)^4$ are given by~(\ref{eq:36}).
From the a priori estimates on $\left \| [Du^{\ell+1}] \right \|_{\infty}$ and $\| [D\tilde u^{\ell+1}] \|_{\infty}$, the inequalities  $1<\beta <2$ and the bound from below on $m$,
 there exists a positive constant $c$  such that
\[   \wtM^\ell_{i,j}  \min \left(| p ^{\ell+1}_{i,j}|_{\infty}^{\beta-2},| \tilde p ^{\ell+1}_{i,j}|_{\infty}^{\beta-2}  \right) \ge c.\]
Hence
\begin{equation}
  \label{eq:53}
 \dt \, h^2 \sum_{\ell=0}^{N_T-1}  \sum_{i,j}   |p ^{\ell+1}_{i,j}-\tilde p ^{\ell+1}_{i,j}|^2 =o(1).
\end{equation}
 It is easy to see from the periodicity that (\ref{eq:53}) implies
 \begin{equation}
   \label{eq:54}
 \dt \, h^2 \sum_{\ell=0}^{N_T-1}  \sum_{i,j}  \left| [D_h u^{\ell+1}]_{i,j}- [D_h \tilde u^{\ell+1}]_{i,j}\right |^2 =o(1).
 \end{equation}
\paragraph{Step 2: a priori estimate on $(m^n_{i,j})$}
Multiplying the  second equation in (\ref{eq:21}) by $m^n_{i,j}$, summing with respect to the indices  $i$ and $j$ and $n$, we get that
\begin{equation}  \label{eq:55}
  \sum_{n=0}^{N_T-1}  \sum_{i,j}    m^{n}_{i,j} \frac {m^{n+1}_{i,j}- m^{n}_{i,j}} {\dt} -\nu  \sum_{n=0}^{N_T-1}  \|\nabla _h m^{n}\|^2_2
  -   \sum_{n=0}^{N_T-1}  \sum_{i,j}    m^{n}_{i,j}  g_q(x_{i,j}, [D_h u^{n+1}]_{i,j}) \cdot  [D_h m^{n}]_{i,j} =0.
\end{equation}
In a very classical way, using the continuity of $g_q$ and  the a priori bound on $\| D_h u^{n+1}\|_{L^{\infty}(\T_h^2)}$ leads to the existence of a positive constant $C$ such that
\begin{equation}\label{eq:56}
  \max_n \|m^n\|^2_{2} +  \dt \sum_{n=0}^{N_T-1} \left \| [D_h m^n] \right \|_2^2 \le C \|m^{N_T}\|_2^2.
\end{equation}
\begin{remark}\label{sec:step-2:-priori}
From~(\ref{eq:21}), we see that for all grid function $(r_{i,j})$ on $\T_h$,
\begin{equation}\label{eq:59}
  \begin{split}
  &\sum_{n=0}^{N_T-1} \sum_{i,j}    r^{n}_{i,j} \frac {m^{n+1}_{i,j}- m^{n}_{i,j}} {\dt} \\ = & \nu  \sum_{n=0}^{N_T-1}  \sum_{i,j}  (\nabla_h m^{n})_{i,j}\cdot  (\nabla_h r^n)_{i,j}
  +   \sum_{n=0}^{N_T-1}  \sum_{i,j}    m^{n}_{i,j}  g_q(x_{i,j}, [D_h u^{n+1}]_{i,j}) \cdot  [D_h r^n]_{i,j}.
  \end{split}
\end{equation}
From the  a priori bound on $\| D_h u^{n+1}\|_{L^{\infty}(\T_h^2)}$ and (\ref{eq:56}), we infer that
\begin{displaymath}
  \sup_{(r^n)_{n}} \frac {   \sum_{n=0}^{N_T-1} \sum_{i,j}    r^{n}_{i,j} \frac {m^{n+1}_{i,j}- m^{n}_{i,j}} {\dt}  }{  \sqrt{\sum_{n=0}^{N_T-1}  \sum_{i,j} | [D_h r^n]_{i,j}|^2  }}\le C.
\end{displaymath}
\end{remark}

\paragraph{Step 3}
From (\ref{eq:56}), we see  that the family of functions $(m_h)$ is bounded in $L^2(0,T; H^1(\T^2))$. Moreover, from (\ref{eq:59}) and the  a priori bound on $\| D_h u^{n+1}\|_{L^{\infty}(\T_h^2)}$,
 we can use the same arguments as in e.g. \cite{MR1804748} pages 855-858 and prove that the family of functions $(m_h)$ has the following property: there exists a constant $C$ such that
 \begin{equation}
   \label{eq:60}
  \|m_h( \cdot-\tau,\cdot) - m_h(\cdot, \cdot)\|^2_{L^2( (\tau,T)\times \T^2)} \le C \tau     h^2 \dt \sum_{n=0}^{N_T-1} \left( \sum_{i,j} | [D_h m^n]_{i,j}|^2  + \sum_{i,j} (m^n_{i,j})^2\right),
 \end{equation}
for all $\tau\in (0,T)$.
Since the right hand side is bounded by $C\tau   h^2 \|m^{N_T}\|_2^2$,  Kolmogorov's theorem (see e.g. \cite{MR697382}, \cite{MR1652593}, \cite{MR1804748} page 833) implies that the family of  functions $(m_h)$
is relatively compact in $L^2( (0,T)\times \T^2)$:
we can extract a subsequence  of parameters $h$ and $\dt$ tending to $0$ such that $ m_h$ converges to $\bar m$ strongly in $L^2( (0,T)\times \T^2)$, and (\ref{eq:60}) holds for $\bar m$.
\\
Therefore, from ($\Phi_{h4}$),
\begin{displaymath}
 \lim_{h,\dt\to 0} \dt \sum_{n=0}^{N_T} \|\Phi[\bar m (t_n,\cdot)] - \Phi_h[m^n]\|^2_{L^\infty(\T^2_h)} =0.
\end{displaymath}
On the other hand, from (\ref{eq:45}), we see that
\[\ds  \dt \, h^2  \sum_{n=0}^{N_T-1} \left(\Phi_h[m^{n}]-\Phi_h[\wtM^{n}],m^{n} -\wtM^{n} \right)_2 =o(1).\]
Then, using (\ref{eq:17}), we deduce from the previous two formulas that
\begin{displaymath}
 \lim_{\dt\to 0} \dt \sum_{n=0}^{N_T}   \int_{\T_h^2}  \left( \Phi[ m (t_n,\cdot)] (x)- \Phi[\bar m  (t_n,\cdot) ] (x)\right)   (m(t_n,x)-\bar m(t_n,x)) dx =0,
\end{displaymath}
which implies that
\begin{displaymath}
  \int_0^T  \int_{\T_h^2}  \left( \Phi[ m (t,\cdot)] (x)- \Phi[\bar m  (t,\cdot) ] (x)\right)   (m(t,x)-\bar m(t,x)) dxdt =0,
\end{displaymath}
The monotonicity of $\Phi$ then  implies that  $m=\bar m$. From the uniqueness of the limit $\bar m$,  we have proven that the whole
 family $m_h$ converges to $m$ in  $L^2((0,T)\times \T^2)$ as $h$ and $\dt$ tend to zero.
\\
We conclude as in Step 3 of the proof of Theorem \ref{sec:case-when-hx-4} for the convergence of $u_h$ to $u$ in $\cC^0[0,T]\times \T^2)$ and  in $ L^2(0,T; W^{1,2}(\T^2))$.
\end{proof}

\section{Study of the convergence in the case when $\Phi$ is a local  operator}
\label{sec:study-conv-case-1}

\subsection{A priori estimates for  (\ref{eq:21})-(\ref{eq:22}) with local operators $\Phi_h$}
\label{sec:stab-estim-refeq:16}
We have a result similar to Theorem 2.7 in \cite{MR2295621}:
\begin{lemma}
  \label{sec:stab-estim-refeq:19}
Assume that $0\le m_T(x)\le \bar m_T$ and that $ u_0$ is a continuous function.
 If $g$ is given by (\ref{eq:10})-(\ref{eq:11}), $(\Phi_h[m])_{i,j} = F( m_{i,j})$, where
\begin{description}
\item{($F_1$)}  $F$ is a $\cC^0$ function defined on $[0,\infty)$
\item{($F_2$)} there exist three constants $\delta>0$ and
  $\gamma>1$ and $C_1\ge 0$ such that
\begin{equation*}
mF(m)  \ge \delta |F(m) |^\gamma -C_1,\quad \forall m\ge 0,
\end{equation*}
\end{description}
 then there exists two constants  $c$ and $C>0$ such that
 \begin{itemize}
 \item  $ u_{i,j}^n\ge c$, for all $n$, $i$ and $j$
 \item
   \begin{equation}
     \label{eq:42}
     h^2\dt \sum_{n=1}^{N_T} \sum_{i,j}    \Bigl| [D_h u^n ]_{i,j}\Bigr|^\beta  +  h^2\dt \sum_{n=0}^{N_T-1} \sum_{i,j}  \left|F(m^n_{i,j})\right|^{\gamma} \le C
   \end{equation}
\item
  \begin{equation}
    \label{eq:43}
\max_{0\le n\le N_T}  h^2  \sum_{i,j}  |u^n_{i,j}|\le C
  \end{equation}
\item Finally, let us call $U^n$ the sum $ h^2 \sum_{i,j} u^n_{i,j}$ and $U_{h}$ the piecewise linear function  obtained by interpolating the values $U^n$ at the
points $(t_n)$: the family of functions $(U_{h} )$ is bounded in $W^{1,1}(0,T)$ by a constant independent of $h$ and $\dt$.
\end{itemize}
\end{lemma}
\begin{proof}

From the two assumptions on $F$, we deduce that $\underline F\equiv \inf_{m\in \R_+} F(m)$  is a real number and that $\underline{F}=  \min_{m\ge 0} F(m)$. Note that $\underline F= F(0)$ if $F$ is nondecreasing.\\
A standard comparison argument shows that
\begin{displaymath}
  u_{i,j}^n \ge \min_{x\in \T^2} u^0(x) + \left(\underline{F}- \max_{x\in \T^2} \ch(x)\right)  t_n   \ge    \min_{x\in \T^2} u^0(x) - T \left(\underline{F}- \max_{x\in \T^2} \ch(x)\right)^- ,
\end{displaymath}
so $ u_{i,j}^n$ is bounded from below by a constant independent of $h$ and $\dt$.
\\
Consider $\tilde u_{i,j}^n= n \dt F(\bar m_T)$ and $\tilde m_{i,j}^n= \bar m_T$ for all $i,j,n$.
We have \begin{displaymath}
\left\{
  \begin{array}[c]{llr}
\frac {\tilde u^{n+1}_{i,j}- \tilde u^{n}_{i,j}} {\dt}  -\nu (\Delta_h \tilde u^{n+1})_{i,j}
 + g(x_{i,j}, [D_h \tilde u^{n+1}]_{i,j} )&=F(\bar m_T) + \ch(x_{i,j}) , \\
\frac {\tilde m^{n+1}_{i,j}- \tilde m^{n}_{i,j}} {\dt} +\nu (\Delta_h \tilde m^{n})_{i,j}
+\cT_{i,j}(\tilde u^{n+1},\tilde m^n)&=0.
\end{array}\right.
\end{displaymath}
Identity (\ref{eq:41}) becomes
 \begin{equation}
\label{eq:44}
  \begin{split}
& h^2\dt \cG(m, u,\tilde u) + h^2\dt \cG(\tilde m ,\tilde u,u) + h^2\dt  \sum_{n=0}^{N_T-1}
   \sum_{i,j}     (F(m_{i,j}^n) -F (\bar m_T)   (m_{i,j}^n-\bar m_T ) \\
= & h^2 \dt \sum_{n=0}^{N_T-1}  \sum_{i,j}    \ch(x_{i,j})  (m_{i,j}^n-\tilde m_{i,j}^n ) +
  h^2(m^{N_T} -\bar m_T, u^{N_T}  - T\, F(\bar m_T) )_2 -    h^2(m^{0} -  \bar m_T, u^{0} )_2.
  \end{split}\end{equation}
Note that
\begin{displaymath}
  \cG(m, u,\tilde u) = \cG(m, u,0)  \quad  \hbox{and}\quad   \cG(\tilde m ,\tilde u,u)=\cG(\tilde m ,0,u),
\end{displaymath}
because $\tilde u_{i,j}^n$ does not depend on $i,j$.
On the other hand,
\begin{enumerate}
\item Since the function $\ch$ is bounded, and $m^n$ is a discrete probability density, there exists a constant $C$ such that
  \begin{displaymath}
  \left| h^2 \dt \sum_{n=0}^{N_T-1}  \sum_{i,j}    \ch(x_{i,j})  (m_{i,j}^n-\tilde m_{i,j}^n )\right| \le C  .
  \end{displaymath}
\item Since $m^{N_T}-\bar m_T$ is nonpositive with a bounded mass,  and since $u^n$ is bounded from below by a constant,  there exists a constant $C$ such that
  \begin{displaymath}
     h^2(m^{N_T} -\bar m_T, u^{N_T}  - T\, F(\bar m_T) )_2 \le C.
  \end{displaymath}
\item Since $u_0$ is continuous on $\T^2$ and $m_0$ is a discrete probability density,  there exists a constant $C$ such that
  \begin{displaymath}
    -    h^2(m^{0} -  \bar m_T, u^{0} )_2 \le C.
  \end{displaymath}
\item Finally, we know that
  \begin{displaymath}
  (F(m_{i,j}^n)-F(\bar m_T)) (m_{i,j}^n-\bar m_T) \ge  \delta \left|F(m_{i,j}^n\right |^\gamma -C_1
- \bar m_T  F(m_{i,j}^n) -m_{i,j}^n F(\bar m_T) +\bar m_TF(\bar m_T).
\end{displaymath}
Moreover, since $\gamma>1$, there exists two constants  $c=  \frac {\delta} 2$ and $C$  such that
 $\delta \left| F(m_{i,j}^n)\right|^\gamma - \bar m_T  F(m_{i,j}^n) \ge  c  \left|F(m_{i,j}^n) \right|^\gamma - C$.
Since $m^n\in \cK_h$, summing yields that for a possibly different constant $C$,
\begin{displaymath}
  h^2 \sum_{i,j}  (F(m_{i,j}^n)-F(\bar m_T)) (m_{i,j}^n-\bar m_T) \ge  c h^2 \sum_{i,j}   \left|F(m_{i,j}^n)\right|^\gamma -C.
\end{displaymath}
\end{enumerate}
In the case $\beta\ge 2$, we get (\ref{eq:42})  from (\ref{eq:44}), from  (\ref{eq:37})  and from the four points above.
In the case $1<\beta<2$, we get (\ref{eq:42})  from (\ref{eq:44}),  from Remark \ref{sec:basic-identity-}   and from the four points above.\\
Finally, summing the first equation in (\ref{eq:21}) for all $i,j$, $0 \le \ell <n$  one gets that
\begin{displaymath}
  h^2\sum_{i,j} u^{n}_{i,j}  + h^2\dt \sum_{\ell=0}^{q-1}  \sum_{i,j}  g(x_{i,j}, [D_h u^{\ell+1}]_{i,j} )=  h^2 \dt \sum_{\ell=0}^{n-1}  \sum_{i,j} F( m^{\ell} _{i,j}) +  h^2\sum_{i,j} u^{0}_{i,j} .
\end{displaymath}
Using (\ref{eq:42}), we get that there exists a constant $C$ such that
\begin{displaymath}
  h^2\sum_{i,j} u^{n}_{i,j}  \le C,
\end{displaymath}
and since $ u^{n}_{i,j}$ is bounded from below by a  constant, we get (\ref{eq:43}).\\
Finally, remember that $U^n$ is the sum $ h^2 \sum_{i,j} u^n_{i,j}$; summing the first equations in (\ref{eq:21}) for all $i,j$, we obtain that
\begin{displaymath}
  \frac {U^{n+1}- U^n} \dt =    G^{n+1} \equiv -h^2 \sum_{i,j} \left(g(x_{i,j}, [D_h u^{n+1}]_{i,j} ) +F (m^{n}_{i,j})\right).
\end{displaymath}
The a priori estimate (\ref{eq:42}) implies that $\dt \sum_{n=0}^{N_T=1} |G^{n+1}|$ is bounded by a constant.
 This implies that the piecewise linear function $U_h$ obtained by interpolating the values $U^n$ at the
points $(t_n)$ is bounded in $W^{1,1}(0,T)$ by a constant independent of $h$ and $\dt$.
\end{proof}


\subsection{Convergence theorems}\label{sec:convergence-theorems}
The case when $\Phi$ is a local operator, i.e. $\Phi[m](x)=F(m(x))$ brings additional difficulties, because there is no a priori Lipschitz estimates on $u_h$: such estimates were
 used several times in the proofs of Theorems \ref{sec:case-when-hx-4} and \ref{sec:case-when-1beta}. \\
For simplicity, we are going to make the assumption that the continuous problem has a classical solution:
existence of a classical solution can  be true for local operators $\Phi$: for example, it has been proved in \cite{porretta2012} that if $\beta=2$,
and $F$ is $\cC^1$ and bounded from below, and if the functions $u^0$ and $m_T$ are $\cC^2$ then there is a classical solution.
\begin{remark}
  In the case of the stationary problem  (\ref{eq:4}), it can be proved that, if ($F_2$) holds with $\gamma>2$ ($2$ is the space dimension) and $F$ is nondecreasing,
 then  (\ref{eq:4}) has a classical solution  for any $\beta>1$, by using the weak Bernstein method studied in \cite{MR833413}.
\end{remark}

\paragraph{Standing assumptions (in \S~\ref{sec:convergence-theorems})}
\begin{itemize}
\item We take $\nu>0$
\item  The Hamiltonian is of the form  (\ref{eq:7}) and the function $x\to \ch(x)$ is $\cC^1$ on $\T^2$
\item  the functions $u^0$ and $m_T$ are smooth, and  $m_T\in \cK$ is bounded from below by a positive number
\item  We assume that  ($F_1$) and ($F_2$) hold  that   there exists three positive constants $\underline \delta$, $ \eta_1>0$ and $0<\eta_2<1$ such that
  $F'(m)\ge\underline{\delta} \min (m^{\eta_1}, m^{-\eta_2})$
\item We consider a numerical Hamiltonian given by (\ref{eq:10})-(\ref{eq:11})
\end{itemize}

\begin{theorem}\label{sec:study-conv-case-2}
We make the standing assumptions stated above and  we assume furthermore that there is  unique classical solution $(u,m)$ of (\ref{eq:1})-(\ref{eq:3}) such that $m>0$.

Let $u_h$ (resp. $m_h$) be the   piecewise trilinear function  in $\cC([0,T]\times\T^2)$ obtained by interpolating the values $u_{i,j}^n$ (resp $m_{i,j}^n$)  at the nodes of the space-time grid.
The functions $u_h$ converge  in   $L^\beta(0,T; W^{1,\beta} (\T^2))$  to $u$ as $h$ and $\dt$ tend to $0$. The functions $m_h$ converge to $m$
in $L^{2-\eta_2}((0,T)\times \T^2)$ as $h$ and $\dt$ tend to $0$.
 \end{theorem}
 \begin{proof}
Call $\bar m= \max m(t,x)$ and $0<\underline m =\min  m(t,x)$.
   \paragraph{Step 1}
We start from (\ref{eq:45}) where $a$ and $b$
 are the same consistency errors (with the same bounds) as in  \S~\ref{sec:study-conv-case}.
 From Lemma \ref{sec:stab-estim-refeq:19}, the a priori bound (\ref{eq:43}) holds for $u_h$. This implies that  $ \lim_{h, \dt\to 0}   h^2  \max_n |( b^{n-1}, u^n-\tilde u^n )_2| =0$.
 From the fact that $m^n\in \cK_h$, we also get that $  \lim_{h, \dt\to 0}   h^2  \max_n| ( a^n, m^n-\tilde m^n )_2|=0$.
\\
Therefore, if $\beta\ge 2$, we obtain (\ref{eq:46}) and (\ref{eq:47}).
\\
If $1<\beta<2$,   we are going to prove that (\ref{eq:47}) also holds: we have
\begin{displaymath}
      2^{\beta -3} \beta(\beta-1)        \dt \, h^2 \sum_{\ell=0}^{N_T-1}  \sum_{i,j}   \wtM^\ell_{i,j}     1_{ p ^{\ell+1}_{i,j} +\tilde p ^{\ell+1}_{i,j}\not=0 }
 \min \left(| p ^{\ell+1}_{i,j}|_{\infty}^{\beta-2},| \tilde p ^{\ell+1}_{i,j}|_{\infty}^{\beta-2}  \right)
 |p ^{\ell+1}_{i,j}-\tilde p ^{\ell+1}_{i,j}|^2 =o(1).
\end{displaymath}
where  $ p^{\ell+1}_{i,j},\tilde p^{\ell+1}_{i,j} \in (\R_+)^4$ are given by~(\ref{eq:36}).
Let us define for brevity  \[ \ds  V ^{\ell+1}_{i,j} =    2^{\beta -3} \beta(\beta-1)   1_{ p ^{\ell+1}_{i,j} +\tilde p ^{\ell+1}_{i,j}\not=0 }
 \min \left(| p ^{\ell+1}_{i,j}|_{\infty}^{\beta-2},| \tilde p ^{\ell+1}_{i,j}|_{\infty}^{\beta-2}  \right)
 |p ^{\ell+1}_{i,j}-\tilde p ^{\ell+1}_{i,j}|^2.\]
Assume that $\tilde p_{i,j}^{\ell+1} \not =0$. We have
\begin{displaymath}
  V ^{\ell+1}_{i,j} \ge    2^{\beta-3} \beta (\beta-1)  \min( |p^{\ell +1 }_{i,j} | ^{\beta-2},|\tilde p^{\ell +1 }_{i,j} |^{\beta-2}) |p^{\ell +1 }_{i,j} -\tilde p^{\ell +1 }_{i,j} |^2.
\end{displaymath}
  \begin{itemize}
  \item If $|p^{\ell +1 }_{i,j} |\le |\tilde p^{\ell +1 }_{i,j} |$, then
    \begin{displaymath}
        V ^{\ell+1}_{i,j} \ge   2^{\beta-3} \beta (\beta-1)  |\tilde p^{\ell +1 }_{i,j} |^{\beta-2}  |p^{\ell +1 }_{i,j} -\tilde p^{\ell +1 }_{i,j} |^2.
    \end{displaymath}
  \item  If $|\tilde p^{\ell +1 }_{i,j} |\le |p^{\ell +1 }_{i,j} |$ and $|p^{\ell +1 }_{i,j}  -\tilde p^{\ell +1 }_{i,j} |\le |\tilde p^{\ell +1 }_{i,j}  | /2$, then
 \begin{displaymath}
   \begin{split}
        V ^{\ell+1}_{i,j} & \ge   2^{\beta-3} \beta (\beta-1)  |p^{\ell +1 }_{i,j} |^{\beta-2}  |p^{\ell +1 }_{i,j} -\tilde p^{\ell +1 }_{i,j} |^2\\
& \ge  2^{\beta-3} \beta (\beta-1) (3 |\tilde p^{\ell +1 }_{i,j} |/2)^{\beta-2}  |p^{\ell +1 }_{i,j} -\tilde p^{\ell +1 }_{i,j} |^2 \\
&=   \frac { 3^{\beta -2} } 2  \beta (\beta-1) |\tilde p^{\ell +1 }_{i,j} |^{\beta-2}  |p^{\ell +1 }_{i,j} -\tilde p^{\ell +1 }_{i,j} |^2.
   \end{split}
    \end{displaymath}
\item  If $|\tilde p^{\ell +1 }_{i,j} |\le |p^{\ell +1 }_{i,j} |$ and $|p^{\ell +1 }_{i,j}  -\tilde p^{\ell +1 }_{i,j} |\ge |\tilde p^{\ell +1 }_{i,j}  | /2$, then
 \begin{displaymath}
   \begin{split}
        V ^{\ell+1}_{i,j} & \ge   2^{\beta-3} \beta (\beta-1)  |p^{\ell +1 }_{i,j} |^{\beta-2}  |p^{\ell +1 }_{i,j} -\tilde p^{\ell +1 }_{i,j} |^2\\
& \ge  2^{\beta-3} \beta (\beta-1) ( |p^{\ell +1 }_{i,j} -\tilde p^{\ell +1 }_{i,j} | + |\tilde p^{\ell +1 }_{i,j} |) ^{\beta-2}  |p^{\ell +1 }_{i,j} -\tilde p^{\ell +1 }_{i,j} |^2\\
&  \ge  2^{\beta-3} \beta (\beta-1) ( 3|p^{\ell +1 }_{i,j} -\tilde p^{\ell +1 }_{i,j} |)^{\beta-2}   |p^{\ell +1 }_{i,j} -\tilde p^{\ell +1 }_{i,j} |^2\\
&=   \frac { 6^{\beta -2} } 2  \beta (\beta-1)   |p^{\ell +1 }_{i,j} -\tilde p^{\ell +1 }_{i,j} |^\beta.
   \end{split}
    \end{displaymath}
  \end{itemize}
If $\tilde p^{\ell +1 }_{i,j} =0$ and $p^{\ell +1 }_{i,j} \not = 0$, then $  V ^{\ell+1}_{i,j} \ge   2^{\beta-3} \beta (\beta-1)  |p^{\ell +1 }_{i,j} |^{\beta}$.
\\
From the observation above, and since $  \wtM^\ell_{i,j} \ge \underline m$,
\begin{displaymath}
    \dt \, h^2 \sum_{\ell=0}^{N_T-1}  \sum_{i,j}    \left(
      \begin{array}[c]{l}
  \ds      1_{\{  |p^{\ell +1 }_{i,j} |\le |\tilde p^{\ell +1 }_{i,j} | \}} |\tilde p^{\ell +1 }_{i,j} |^{\beta-2}  |p^{\ell +1 }_{i,j} -\tilde p^{\ell +1 }_{i,j} |^2\\
\ds+        1_{\{  |p^{\ell +1 }_{i,j} |\ge |\tilde p^{\ell +1 }_{i,j} | \}}   1_{\{|p^{\ell +1 }_{i,j}  -\tilde p^{\ell +1 }_{i,j} |\le |\tilde p^{\ell +1 }_{i,j}  | /2\}}
   |\tilde p^{\ell +1 }_{i,j} |^{\beta-2}  |p^{\ell +1 }_{i,j} -\tilde p^{\ell +1 }_{i,j} |^2\\
+ \ds   1_{\{  |p^{\ell +1 }_{i,j} |\ge |\tilde p^{\ell +1 }_{i,j} | \}}   1_{\{|p^{\ell +1 }_{i,j}  -\tilde p^{\ell +1 }_{i,j} |\ge |\tilde p^{\ell +1 }_{i,j}  | /2\}} |p^{\ell +1 }_{i,j} -\tilde p^{\ell +1 }_{i,j} |^\beta.
      \end{array}\right) =o(1).
\end{displaymath}
Note also that from the regularity of $u$,  $ |\tilde p^{\ell +1 }_{i,j} |^{\beta-2}$ is bounded from below by a constant independent of $h$, $\dt$, $i,j,\ell$.
Thus
\begin{displaymath}
    \dt \, h^2 \sum_{\ell=0}^{N_T-1}  \sum_{i,j}    \left(
      \begin{array}[c]{l}
  \ds      1_{\{  |p^{\ell +1 }_{i,j} |\le |\tilde p^{\ell +1 }_{i,j} | \}}  |p^{\ell +1 }_{i,j} -\tilde p^{\ell +1 }_{i,j} |^2\\
\ds+        1_{\{  |p^{\ell +1 }_{i,j} |\ge |\tilde p^{\ell +1 }_{i,j} | \}}   1_{\{|p^{\ell +1 }_{i,j}  -\tilde p^{\ell +1 }_{i,j} |\le |\tilde p^{\ell +1 }_{i,j}  | /2\}}
    |p^{\ell +1 }_{i,j} -\tilde p^{\ell +1 }_{i,j} |^2\\
+ \ds   1_{\{  |p^{\ell +1 }_{i,j} |\ge |\tilde p^{\ell +1 }_{i,j} | \}}   1_{\{|p^{\ell +1 }_{i,j}  -\tilde p^{\ell +1 }_{i,j} |\ge |\tilde p^{\ell +1 }_{i,j}  | /2\}} |p^{\ell +1 }_{i,j} -\tilde p^{\ell +1 }_{i,j} |^\beta.
      \end{array}\right) =o(1).
\end{displaymath}
Using a  H{\"o}lder inequality, we deduce that
\begin{displaymath}
  h^2 \dt \sum_{\ell =1}^{N_T}    \sum_{i,j}     \Bigl| p^{\ell+1}_{i,j} - \tilde p^{\ell+1}_{i,j}\Bigr|^\beta =o(1),
\end{displaymath}
and finally (\ref{eq:47}).

 \paragraph{Step 2}
We also obtain from  (\ref{eq:45}) that
\begin{equation}
  \label{eq:57}
   h^2\dt  \sum_{n=0}^{N_T-1}   \sum_{i,j}     (F(m_{i,j}^n) -F (\tilde m_{i,j}^n )   (m_{i,j}^n-\tilde m_{i,j}^n ) =o(1).
\end{equation}
We split the sum w.r.t. $ (i,j)$ in the left hand side of (\ref{eq:57}) into
\begin{displaymath}
  \begin{split}
    S^n_1=  \sum_{i,j}  ( (m_{i,j}^n-\tilde m_{i,j}^n )^-)^2  \int_0^1 F'(\tilde m_{i,j}^n + t  (m_{i,j}^n-\tilde m_{i,j}^n ) )  dt,\\
    S^n_2= \sum_{n=0}^{N_T-1}   \sum_{i,j}  ( (m_{i,j}^n-\tilde m_{i,j}^n )^+)^2  \int_0^1 F'(\tilde m_{i,j}^n + t  (m_{i,j}^n-\tilde m_{i,j}^n ) )  dt .
  \end{split}
\end{displaymath}
Call $\bar m= \max m(t,x)$ and $\underline m =\min  m(t,x)>0$;
 there exists a positive number $c$ depending on $\underline m$ and $\bar m$ but independent of $h$ and $\delta$, and $(i,j, n)$ such that
\begin{displaymath}
  \begin{split}
    S^n_1&\ge   c \sum_{i,j}  ( (m_{i,j}^n-\tilde m_{i,j}^n )^-)^2  \int_0^1 (\tilde m_{i,j}^n + t  (m_{i,j}^n-\tilde m_{i,j}^n ) )^{\eta_1}  dt
\\ &= \frac c {\eta_1+1} \sum_{i,j}  ( (m_{i,j}^n-\tilde m_{i,j}^n )^-)    ((\tilde m_{i,j}^n) ^ {\eta_1+1} -(m_{i,j}^n )^{\eta_1+1})
\\ & \ge \frac c {\eta_1+1} \sum_{i,j}  ( (m_{i,j}^n-\tilde m_{i,j}^n )^-)^2  (\tilde m_{i,j}^n )^{\eta_1}.
\end{split}
\end{displaymath}
The latter inequality comes from the nondecreasing character of the function $\chi: [0,y]\mapsto \R$, $\chi(z)= \frac {y^{\eta_1+1} -z^{\eta_1+1}} {y-z}$. Thus,
$\chi(z)\ge \chi(0) =y^{\eta_1}$.
Hence, there exists a constant $c$ depending on the bounds on the density $m$ solution of (\ref{eq:1})-(\ref{eq:3})  but not on  $h$ and $\dt$, and $(i,j, n)$ such that
\begin{displaymath}
   S^n_1\ge  c \sum_{i,j}  ( (m_{i,j}^n-\tilde m_{i,j}^n )^-)^2 .
\end{displaymath}
On the other hand
    \begin{displaymath}
  \begin{split}
 S^n_2&\ge   c \sum_{i,j}  ( (m_{i,j}^n-\tilde m_{i,j}^n )^+)^2  \int_0^1 (\tilde m_{i,j}^n + t  (m_{i,j}^n-\tilde m_{i,j}^n ) )^{-\eta_2}  dt
\\ &= \frac c {1-\eta_2} \sum_{i,j}  ( (m_{i,j}^n-\tilde m_{i,j}^n )^+)    (( m_{i,j}^n) ^ {1-\eta_2} -(\tilde m_{i,j}^n )^{1-\eta_2})
\\& \ge  c \sum_{i,j}  ( (m_{i,j}^n-\tilde m_{i,j}^n )^+)^2    ( m_{i,j}^n) ^ {-\eta_2} .
  \end{split}
\end{displaymath}
But there exists a constant $c$ such that for all $ y\in [\underline m,\bar m]$: if  $z\ge y+1$
\begin{displaymath}
   (z-y)^2 z^{-\eta_2} \ge   (z-y)^{2-\eta_2} \inf_{z\ge y+1} \frac {(z-y)^{\eta_2}} {z^{\eta_2}} \ge c  (z-y)^{2-\eta_2},
\end{displaymath}
and if $y\le z\le y+1$,
\begin{displaymath}
   (z-y)^2 z^{-\eta_2}  \ge c  (z-y)^{2}.
\end{displaymath}
Therefore there exists a constant $c$ such that
\begin{displaymath}
  S_1^n+S_2^n \ge c\left( \sum_{i,j}  (m_{i,j}^n-\tilde m_{i,j}^n )^2 1_{\{ m_{i,j}^n\le \tilde m_{i,j}^n+1 \}}
+\sum_{i,j}  (m_{i,j}^n-\tilde m_{i,j}^n )^{ 2-\eta_2}  1_{\{ m_{i,j}^n\ge \tilde m_{i,j}^n+1 \}}  \right).
\end{displaymath}
Then (\ref{eq:57}) implies that
\begin{displaymath}
  \lim_{h,\dt\to 0} h^2 \dt \sum_{n=0}^{N_T-1} \left( \sum_{i,j}  (m_{i,j}^n-\tilde m_{i,j}^n )^2 1_{\{ m_{i,j}^n\le \tilde m_{i,j}^n+1 \}}
+\sum_{i,j}  (m_{i,j}^n-\tilde m_{i,j}^n )^{ 2-\eta_2}  1_{\{ m_{i,j}^n\ge \tilde m_{i,j}^n+1 \}}  \right) =0.
\end{displaymath}
Then, a H{\"o}lder inequality leads to
\begin{displaymath}
  \lim_{h,\dt\to 0} h^2 \dt \sum_{n=0}^{N_T-1}  \sum_{i,j}  |m_{i,j}^n-\tilde m_{i,j}^n|^{ 2-\eta_2} =0.
\end{displaymath}
 \paragraph{Step 3}
From the previous two steps, up to an extraction of a sequence,  $m_h\to m$ in $L^{2-\eta_2}((0,T)\times \T^2)$ and almost everywhere in $(0,T)\times \T^2$,
 $\nabla u_h$ converges to $\nabla u$ strongly in $L^\beta ((0,T)\times \T^2)$.  Moreover,
 from the last point in Lemma~\ref{sec:stab-estim-refeq:19}, the sequence of piecewise linear functions $(U_h)$ on $[0,T]$ obtained by interpolating the values
$U^n=h^2\sum_{i,j} u_{i,j}^n$ at the points $(t_n)$
 is bounded in $W^{1,1}(0,T)$, so up to a further extraction of a subsequence, it converges to some function $U$ in $L^\beta(0,T)$. As a result, there exists a
 function $\psi$ of the variable $t$ such that  $u_h\to u+\psi$ in $L^\beta(0,T; W^{1,\beta} (\T^2))$.
\\
From the a priori estimate (\ref{eq:42}), the sequence $(F(m_h))$ is bounded in $L^\gamma((0,T)\times \T^2)$ for some $\gamma>0$, which implies that it is uniformly integrable on
$(0,T)\times \T^2$. On the other hand, $F(m_h)$ converges almost everywhere to $F(m)$. Therefore, from Vitali's theorem, see e.g. \cite{0925.00005}
,  $F(m_h)$ converges  to $F(m)$ in $L^1((0,T)\times \T^2)$, (in fact, it is also possible to show that
$F(m_h)$ converges  to $F(m)$ in $L^q((0,T)\times \T^2)$ for all $q\in [1,\gamma)$).
\\
It is then possible to pass to the limit in the discrete Bellman equation, which yields  that $\frac {\partial \psi}{\partial t}=0$
in the sense of distributions in $(0,T)$. Hence $\psi$ is a constant.
\\
We are left with proving that $\psi$ is indeed $0$. For that, we split  $ \frac {\partial u_h}{\partial t}$ into the sum $ \mu_h +\eta_h$, where
\begin{itemize}
\item $\mu_h|_{t\in (t_n,t_{n+1}]}$ is constant w.r.t. $t$ and  piecewise linear w.r.t. $x$, and takes the value $\nu (\Delta _h u^{n+1})_{i,j}$ at the node $\xi_{i,j}$
\item $\eta_h$ is the remainder, see (\ref{eq:18}).  This term is constructed by interpolating  the values $F(m^n_{i,j})-g(x_{i,j}, [D_h u^{n+1}]_{i,j})$ at the grid nodes.
\end{itemize}
From the observations above,  $(\eta_h)$ converges in $L^1((0,T)\times \T^2)$, (because of the strong convergence of $\nabla u_h$ and of $F(m_h)$). On the other hand, from (\ref{eq:47}), it is not difficult to see that
$(\mu_h)$ is a Cauchy sequence in $L^\beta(0,T; (W^{s,\beta/(\beta-1)} (\T^2)  )')$ for $s$ large enough,
 (here $(W^{s,\beta/(\beta-1)} (\T^2)  )'$ is the topological dual of $W^{s,\beta/(\beta-1)} (\T^2)$).
Hence, $ ( \frac {\partial u_h}{\partial t}) $ converges in $L^1(0,T; (W^{s,\beta/(\beta-1)} (\T^2)  )')$. Therefore,
 $u_h$ converges in $\cC^0([0,T];  (W^{s,\beta/(\beta-1)} (\T^2)  )')$; since $(u_h(t=0))$ converges to $u_0$, we see that $\psi=0$.
\\
This implies that the extracted sequence $u_h$ converges to $u$ in $L^\beta(0,T; W^{1,\beta} (\T^2))$.
Since the limit is unique, the whole family $(u_h)$ converges to $u$ in  $L^\beta(0,T; W^{1,\beta} (\T^2))$ as $h$ and $\dt$ tend to $0$.
 \end{proof} \\
We give  the corresponding theorem in the ergodic case, without proof, because it is quite similar to that of Theorem~\ref{sec:study-conv-case-2}.
\begin{theorem}\label{sec:step-3}
We make the standing assumptions stated at the beginning of \S~\ref{sec:convergence-theorems} and  we assume furthermore
 that there is  unique classical solution $(u,m,\lambda)$ of (\ref{eq:4})-(\ref{eq:5}) such that $m>0$ and $\int_{\T^2} u(x) dx=0$.
\\
Let $u_h$ (resp. $m_h$) be the   piecewise bilinear function  in $\cC(\T^2)$ obtained by interpolating the values $u_{i,j}$ (resp $m_{i,j}$)  at the nodes of  $\T_h^2$,
where  $\left( (u_{i,j}), (m_{i,j}),\lambda_h \right) $ is the unique solution of the following system:\\
 for all $0\le i,j< N_h$, $m_{i,j}\ge 0$,
\begin{equation}
\label{eq:58}
\left\{
  \begin{array}[c]{llr}
  -\nu (\Delta_h u)_{i,j}  + g(x_{i,j}, [D_h u]_{i,j} ) +\lambda_h &=F(m_{i,j}), \\
-\nu (\Delta_h m)_{i,j} -\cT_{i,j}(u,m)&=0, \\
h^2\sum_{i,j} u_{i,j}=0, \quad h^2\sum_{i,j} m_{i,j}&=1.
\end{array}\right.
\end{equation}
As $h$  tends to $0$, the functions $u_h$ converge  in   $W^{1,\beta} (\T^2)$  to $u$, the functions $m_h$ converge to $m$
in $L^{2-\eta_2}(\T^2)$, and $\lambda_h$ tends to $\lambda$.
 \end{theorem}

\appendix
\section{Proofs of some technical lemmas}\label{sec:proofs-some-techn}

\paragraph{Proof of Lemma~\ref{sec:case-when-hx-1}}
For all $r\in \R^4$, we have
  \begin{displaymath}
    \begin{array}[c]{rcl}
\ds      g_q(x, q)  \cdot r&=& \ds G_p(p)\cdot(-1_{q_1<0} r_1, 1_{q_2>0} r_2,-1_{q_3<0} r_3, 1_{q_4>0} r_4)\\
      &=& \ds \beta |p|^{\beta-2}  \left( -p_1 r_1 + p_2 r_2- p_3 r_3+p_4 r_4 \right)
    \end{array}
  \end{displaymath}
Hence,
\begin{displaymath}
  \begin{array}[c]{rcl}
\ds    -g_q(x, q)\cdot (\tilde q -q ) &=& \ds -G_p(p)\cdot \Bigl(-1_{q_1<0} (\tilde q_1 -q_1 ), 1_{q_2>0}  (\tilde q_2 -q_2 ),-1_{q_3<0}  (\tilde q_3 -q_3 ), 1_{q_4>0}  (\tilde q_4 -q_4 )\Bigr)
\\
&=& \ds -\beta |p|^{\beta-2} \Bigl( -p_1 (\tilde q_1 -q_1 ) +p_2 (\tilde q_2 -q_2 ) -p_3  (\tilde q_3 -q_3 ) + p_4 (\tilde q_4 -q_4 )   \Bigr).
  \end{array}
\end{displaymath}
But
\begin{eqnarray*}
  -p_1 (\tilde q_1 -q_1 )= p_1(\tilde p_1 -p_1) - p_1 \tilde q_1^+ \le  p_1(\tilde p_1 -p_1),
\\
  p_2 (\tilde q_2 -q_2 )= p_2(\tilde p_2 -p_2) - p_2 \tilde q_2^-\le  p_2(\tilde p_2 -p_2),
\\
  -p_3 (\tilde q_3 -q_3 )= p_3(\tilde p_3 -p_3) - p_3 \tilde q_3^+\le  p_3(\tilde p_3 -p_3),
\\
  p_4 (\tilde q_4 -q_4 )= p_4(\tilde p_4 -p_4) - p_4 \tilde q_4^-\le  p_4(\tilde p_4 -p_4).
\end{eqnarray*}
Therefore
\begin{equation}\label{eq:28}
   -g_q(x, q)\cdot (\tilde q -q )\ge - \beta |p|^{\beta-2} p\cdot (\tilde p -p) = -G_p(p)\cdot (\tilde p -p),
\end{equation}
and (\ref{eq:27}) follows immediately.
\paragraph{Proof of Lemma~\ref{sec:case-when-hx-2}}
If $\beta\ge 2$, then
  \begin{displaymath}
    \begin{array}[c]{rcl}
\ds      g(x,\tilde q)-g(x, q) -g_q(x, q)\cdot (\tilde q -q )&\ge& \ds G(\tilde p)-G( p) -G_p(p)\cdot (\tilde p -p )
\\
&=& \ds \int_{0}^1 (1-s) G_{pp}( s \tilde p +(1-s) p) (\tilde p -p)\cdot (\tilde p -p) ds \\
&\ge& \ds  \beta |p-\tilde p|^2 \int_{0}^1 (1-s) |s \tilde p +(1-s) p|^{\beta -2} ds ,
    \end{array}
  \end{displaymath}
where the first (resp. second) inequality comes from Lemma \ref{sec:case-when-hx-1} (resp. Lemma \ref{sec:case-when-hx}).
Hence,
  \begin{displaymath}
    \begin{array}[c]{rcl}
\ds      g(x,\tilde q)-g(x, q) -g_q(x, q)\cdot (\tilde q -q )&\ge& \ds  \beta  |p|^{\beta -2} |p-\tilde p|^2 \int_{0}^1 (1-s)^{\beta-1}  ds
\\ &=& \ds  |p|^{\beta -2} |p-\tilde p|^2.
\end{array}
\end{displaymath}
On the other hand,
 \begin{displaymath}
    \begin{array}[c]{rcl}
\ds      g(x,\tilde q)-g(x, q) -g_q(x, q)\cdot (\tilde q -q )&\ge& \ds  \beta  |\tilde p|^{\beta -2} |p-\tilde p|^2 \int_{0}^1 (1-s) s^{\beta-2}  ds
\\ &=& \ds  \frac 1 {\beta-1}  |\tilde p|^{\beta -2} |p-\tilde p|^2.
\end{array}
\end{displaymath}
The last two estimates yield (\ref{eq:29}) because  $1\ge \frac 1 {\beta-1}$, then (\ref{eq:30}).\\ \\
If $1<\beta<2$ and $p+\tilde p\not =0$, then
 \begin{displaymath}
    \begin{array}[c]{rcl}
\ds      g(x,\tilde q)-g(x, q) -g_q(x, q)\cdot (\tilde q -q )&\ge&  \ds \int_{0}^1 (1-s) G_{pp}( s \tilde p +(1-s) p) (\tilde p -p)\cdot (\tilde p -p) ds \\
&\ge& \ds  \beta(\beta-1) |p-\tilde p|^2 \int_{0}^1 (1-s) |s \tilde p +(1-s) p|^{\beta -2} ds ,
    \end{array}
  \end{displaymath}
where the first (resp. second) inequality comes from Lemma \ref{sec:case-when-hx-1} (resp. Lemma \ref{sec:case-when-hx}). But
$|s \tilde p +(1-s) p|\le 2 \max(|p|_\infty, |\tilde p|_\infty)$ and
since $\beta<2$,  $|s \tilde p +(1-s) p|^{\beta -2} \ge 2^{\beta-2}  \min( |p|_\infty ^{\beta-2},|\tilde p|_\infty^{\beta-2})$. Hence,
\begin{displaymath}
    g(x,\tilde q)-g(x, q) -g_q(x, q)\cdot (\tilde q -q ) \ge   2^{\beta-2} \beta(\beta-1)\min( |p|_\infty ^{\beta-2},|\tilde p|_\infty^{\beta-2}) \int_0^1 (1-s)ds,
\end{displaymath}
which yields (\ref{eq:31}).
\paragraph{Proof of Lemma~\ref{sec:case-when-hx-3}}
  We have that
  \begin{displaymath}
    \begin{split}
&    \left(g_q(x, \tilde q) -g_q(x, q)\right)\cdot r\\ = &  \ds \beta    \left( |\tilde p|^{\beta-2}  \left( -\tilde p_1 r_1 + \tilde p_2 r_2- \tilde p_3 r_3+\tilde p_4 r_4 \right)-    |p|^{\beta-2}  \left( -p_1 r_1 + p_2 r_2- p_3 r_3+p_4 r_4 \right)\right).
      \end{split}
      \end{displaymath}
Call $\ell$ the function defined on $(\R_+)^4$ by
\begin{displaymath}
\ell(p)= \beta |p|^{\beta-2}  p\cdot {\rm Diag}(-1,1,-1,1) r  ,
\end{displaymath}
where $ {\rm Diag}(-1,1,-1,1)$ stands for the diagonal matrix in $\R^{4\times 4}$ whose diagonal is $(-1,1,-1,1)$.
We have
 \begin{displaymath}
   \begin{array}[c]{ll}
  \ds  \left(g_q(x, \tilde q) -g_q(x, q)\right)\cdot r = \ell (\tilde p) -\ell (p) = \int_0^1 \ell_p (s\tilde p +(1-s) p)\cdot (\tilde p -p) ds, \quad &\hbox{if } p+\tilde p \not=0, \\
\ds \left(g_q(x, \tilde q) -g_q(x, q)\right)\cdot r =0, \quad &\hbox{if } p+\tilde p =0.
   \end{array}
  \end{displaymath}
But
\begin{displaymath}
  \ell_p(p)= \beta (\beta-2) |p|^{\beta -4}    \left(p\cdot {\rm Diag}(-1,1,-1,1) r\right)  p + \beta  |p|^{\beta-2}   {\rm Diag}(-1,1,-1,1) r, \quad \forall p\not=0.
\end{displaymath}
Hence, if $ p+\tilde p \not=0$,
\begin{equation}\label{eq:33}
  \begin{array}[c]{ll}
& \ds    \left(g_q(x, \tilde q) -g_q(x, q)\right)\cdot r  \\
=&\ds  \beta (\beta-2) \int_0^1 |s\tilde p +(1-s) p|^{\beta-4}  \Bigl(  (s\tilde p +(1-s) p) \cdot {\rm Diag}(-1,1,-1,1) r\Bigr)      \Bigl( (s\tilde p +(1-s) p) \cdot (\tilde p -p) \Bigr) ds\\
 &\ds +  \beta  \int_0^1 |s\tilde p +(1-s) p|^{\beta-2}    \Bigl( (\tilde p -p)\cdot {\rm Diag}(-1,1,-1,1) r\Bigr) ds.
  \end{array}
\end{equation}
Call $I$ (respectively $II$) the first (respectively second) integral in (\ref{eq:33}).
It is clear that
\begin{displaymath}
  |I|\le  |p-\tilde p|  |r| \int_0^1 |s\tilde p +(1-s) p|^{\beta-2} ds \le  \max(|p|^{\beta-2},|\tilde p|^{\beta-2})  |p-\tilde p|  |r|.
\end{displaymath}
We also have
\begin{displaymath}
  |II|\le  \max(|p|^{\beta-2},|\tilde p|^{\beta-2})  |p-\tilde p|  |r|,
\end{displaymath}
 and (\ref{eq:32}) follows from the last two estimates.
\begin{acknowledgement}
The first author  would like to thank O. Gu{\'e}ant for helpful discussions.
\end{acknowledgement}

\bibliographystyle{plain}
\bibliography{MFG}

\end{document}